\documentclass[11pt,a4paper,reqno]{amsart}
\usepackage{amsmath}
\usepackage{amsfonts}
\usepackage[utf8]{inputenc}
\usepackage[english]{varioref}
\usepackage{dcolumn}
\usepackage[height=22cm , width = 16cm , top = 4cm , left = 3cm, a4paper]{geometry}
\usepackage[final]{graphicx}
\usepackage{epsfig}
\usepackage{pstricks}
\usepackage{psfrag}
\usepackage{booktabs}
\usepackage{delarray}
\usepackage{amssymb}
\usepackage{rotating}
\usepackage{subfigure}
\usepackage{layout}
\usepackage{amsthm}
\usepackage{hyperref}
\usepackage{xcolor}
\usepackage{setspace}
\usepackage{colortbl}
\usepackage[format=hang]{caption}
\usepackage{caption}
\usepackage{epsfig}
\usepackage{epstopdf}
\usepackage{booktabs}
\usepackage{siunitx}
\usepackage[noadjust]{cite}
\newtheorem{theorem}{Theorem}
{
{\newcommand{\pd}{\partial}
{
{

\newcommand{\et}{\textit{et al.} }

\newtheorem{thm}{Theorem}[section]

\newtheorem{rem}[thm]{Remark}

\numberwithin{equation}{section}
\newtheorem{example}{Example}[section]
\parindent 0pt

\title[AHPETM for Pure Breakage and Smoluchowski's Coagulation Equation]{\textbf{Elzaki Transform Based Accelerated Homotopy Perturbation Method for Multi-dimensional Smoluchowski's Coagulation and Coupled Coagulation-fragmentation Equations}}

\author[Arora, Kumar and Mammeri]{Gourav Arora$^\dag$, Rajesh Kumar$^{\dag}$ and Youcef Mammeri$^{\dag\dag}$}
\thanks{$^\dag$Department of Mathematics,
Birla Institute of Technology and Science, Pilani, Rajasthan-333031, India (p20190421@pilani.bits-pilani.ac.in).\\
$^\ast$Corresponding author: Department of Mathematics, Birla Institute of Technology and Science, Pilani, Rajasthan-333031, India (rajesh.kumar@pilani.bits-pilani.ac.in).\\
$^{\dag\dag}$Institut Camille Jordan CNRS UMR 5208, Université Jean Monnet, 42100 Saint-Etienne, France 
 (youcef.mammeri@u-picardie.fr).}
\begin{document}
\maketitle
%
	\begin{quote}
		{\textit{ Abstract: 
This article aims to establish a semi-analytical approach based on the homotopy perturbation method (HPM) to find the closed form or approximated solutions for the population balance equations such as Smoluchowski's coagulation, fragmentation, coupled coagulation-fragmentation and bivariate coagulation equations. An accelerated form of the HPM is combined with the Elzaki transformation to improve the accuracy and efficiency of the method. One of the significant advantages of the technique lies over the classic numerical methods as it allows solving the linear and non-linear differential equations without discretization. Further, it has benefits over the existing semi-analytical techniques such as Adomian decomposition method (ADM), optimized decomposition method (ODM), and homotopy analysis method (HAM) in the sense that computation of Adomian polynomials and convergence parameters are not required. The novelty of the scheme is shown by comparing the numerical findings with the existing results obtained via ADM, HPM, HAM and ODM for non-linear coagulation equation. This motivates us to extend the scheme for solving the other models mentioned above. The supremacy of the proposed scheme is demonstrated by taking several numerical examples for each problem. The error between exact and series solutions provided in graphs and tables show the accuracy and applicability of the method. In addition to this, convergence of the series solution is also the key attraction of the work.}}
	\end{quote}
		\section{Introduction}
		Particulate processes have drawn much attention of researchers because of their technological applications in many engineering and natural science disciplines, including granulation, crystallization, activated sludge flocculation, and raindrop generation \cite{briesen2006simulation,majumder2012lattice,nopens2005pbm,ramkrishna2000population,rhodes2008introduction}. The particle size distribution, which represents the amount of a specific size within the system, affects the behavior during processing and the final product's performance. During processing, distinct mechanisms like nucleation, breakage (fragmentation), aggregation (coagulation) or growth may occur. Breakage refers the phenomenon in which a particle divides into two or more particles, while aggregation refers to two particles merging to form a more extensive particle. Thus, the total number of particles increases during the breakage process, whereas it decreases in the aggregation phenomenon as time passes, but the mass remains conserved in both the situations. The scope of the article is limited to the pure breakage, aggregation in single and multi-dimensions as well as coupled aggregation-breakage equations.
The mathematical formulation of pure fragmentation equation \cite{batycky1997theoretical} is given by
\begin{align}\label{breakage_eqn}
\dfrac{\pd u(x,t)}{\pd t}=\int_{x}^{\infty}B(x,y)S(y)u(y,t)dy-S(x)u(x,t),
\end{align}
and the non-linear Smoluchowski's coagulation equation in $1-D$ is provided by, see \cite{yin2014numerical},
\begin{align}\label{agg_eqn}
\frac{\pd u(x,t)}{\pd t}= \frac{1}{2} \int_{0}^{x} K(x-y,y) u(x-y,t) u(y,t) dy-\int_{0}^{\infty} K(x,y) u(x,t) u(y,t) dy,
\end{align}
with the initial condition 
\begin{align}
u(x,0)= f(x).
\end{align}
Here, $u(x,t) \in (0,\infty) \times [0,T]$ represents the number of particles of size $x$ at time $t$, $B(x,y)$ gives the breakage function, i.e., the rate at which the particles of size $y$ break into particles of size $x$ and the rate at which a particle size $y$ is chosen to break is shown by the selection function $S(y)$. Further, the term  $K(x-y,y)$ denotes the rate at which particles of sizes $x-y$ and $y$ merge to form a particle of size $x$. In equations \eqref{breakage_eqn} and \eqref{agg_eqn}, the first integral terms provide the birth of a particle of size $x$ during the process of breakage and aggregation, respectively, while the second terms in both models indicate the death of particle size $x$.       
\par Along with the number density $u(x,t)$, some integral properties, such as moments, grab the attention because of their physical interpretation. The moments for the number density are defined as 
\begin{align*}
\mu_j(t)= \int_{0}^{\infty}x^j u(x,t) dx, \quad j=0,1,2,\cdots.
\end{align*}
The zeroth moment $\mu_0(t)$ defines the total number of particles in the system at time $t$, first moment $\mu_1(t)$ gives the total mass in system and $\mu_2(t)$ gives the energy dissipated by the system.
\par In solid processing, e.g., in foods and pharmaceuticals, product quality is characterized by multiple particle properties,
for example, the volume and composition of aggregating particles. To model such phenomenon, more then one dimensional is required. Therefore, in the following, the bivariate case is considered, i.e., particles (or individual objects) are characterized by two properties, named $x$ and $y$. The two dimensional aggregation is governed by
\begin{align}\label{2DPBE}
		\dfrac{\pd u(x,y,t)}{\pd t}=& \frac{1}{2} \int_{0}^{x}\int_{0}^{y}K(x-x',y-y',x',y')u(x-x',y-y',t)u(x',y',t)dx'dy' \nonumber \\
		-&\int_{0}^{\infty}\int_{0}^{\infty} K(x,x',y,y')u(x,y,t)u(x',y',t)dx'dy',
\end{align}
with the initial condition
\begin{align}
c(x,y,0)=c_0(x,y) \geq 0.
\end{align}
 Due to complexity of these models and unavailability of the analytical solutions (except some simple cases), several numerical and semi-analytical techniques are applied to solve these problems approximately. Numerical schemes to solve breakage equation \eqref{breakage_eqn} and/or coagulation model \eqref{agg_eqn} includes finite element method \cite{bie2018coupling}, quadrature method of moments \cite{marchisio2003quadrature,su2007solution}, finite volume scheme \cite{singh2019new,singh2022finite,filbet2004mass,kumar2014convergence}, fixed pivot element \cite{giri2013convergence}, fast Fourier transformation method \cite{ahrens2018fft}, cell average technique \cite{kumar2006improved}  and reference therein. The drawbacks of the schemes are shown in the potential reliance of these numerical techniques on non-physical assumptions such as discretization, linearization, sets of basis functions, and many others. Recently, several authors have developed interest in semi-analytical approaches to overcome these shortcomings. These series solution techniques offer results without making such assumptions. Some of the available strategies are Taylor polynomial and radial basis functions \cite{ranjbar2010numerical}, Laplace-variational iteration method \cite{hammouch2010laplace}, ADM \cite{singh2015adomian}, HPM  \cite{kaur2019analytical}, optimal homotopy asymptotic method (OHAM) \cite{dutta2018population}, HAM \cite{kaur2022approximate} and ODM \cite{kaushik2022novel}. Interestingly, some of the algorithm provided the closed form series solutions of coagulation equation \eqref{agg_eqn} for the aggregation kernels
 $$K(x,y)=1,x+y  ,xy\text{ and } x^{\frac{2}{3}}+y^{\frac{2}{3}}, $$
with exponential initial condition ($u(x,0)=e^{-x}, e^{-x}/x$), see \cite{singh2015adomian,kaur2019analytical,kaur2022approximate} for more detailed computations. They also dealt with the breakage equation \eqref{breakage_eqn} with the breakage rate $$b(x,y)=\frac{\alpha}{y}\left(\frac{x}{y}\right)^{\alpha-2} \quad \forall\quad  1\leq \alpha \leq 2 \text{ with selection rate } S(x)=x^\alpha$$
having the exponential ($e^{-x}$) and mono disperse $(\delta(x-a))$ being the two different initial conditions. Hammouch and Mekkaoui in \cite{hammouch2010laplace} developed the Laplace-variational iteration method for solving the coagulation equation \eqref{agg_eqn} only for two cases of aggregation kernels, constant ($K(x,y)=1$) and product ($K(x,y)=xy$). Moreover, Hasseine \et in \cite{hasseine2015two} employed ADM and HPM to solve the breakage equation for the kernel $B(x,y)=\frac{12(y-x)}{x^3}$ with selection function $S(x)=x$. Very recently in \cite{kaushik2022novel}, ODM is implemented to solve the coagulation equation using the parameters $$K(x,y)=1,x+y \text{ and } xy \text{ with } u(x,0)=e^{-x}.$$  
\par In the literature, it was observed that ADM, HPM, and HAM provide the closed form solutions, but some drawbacks are observed in these techniques. In \cite{ganji2006application,dehghan2015convection}, it was found that a large number of iterations are required to obtain a more accurate approximation. When dealing with chaotic systems, Chowdhury and Hashimstill \cite{he2005application} found that time, time step, and the number of terms must be handled with extreme caution. Further, Obidat \cite{odibat2020optimized} has drawn attention to various drawbacks of ADM, including its delayed convergence \cite{jiao2002aftertreatment} and inability to handle boundary conditions \cite{he1997new} for solving non-linear PDEs. These shortcoming were avoided by Obidat in \cite{odibat2020optimized}. To overcome these issues, recently, ODM \cite{kaushik2022novel} has been implemented to solve the model, but the accuracy is still maintained only for a small period of time. Recently, HPM is accelerated by approximating the nonlinear term and incorporating the Elzaki transformation for differential equations \cite{jasrotia2022accelerated} in order to improve the accuracy of the truncated solution. Thus, the first aim of this article is to obtain more accurate solutions to the pure breakage and Smoluchowski's coagulation equations by applying the accelerated homotopy perturbation Elzaki transformation method \cite{jasrotia2022accelerated}.
\par For the second task of this work, combined aggregation-breakage equation is considered which is an intriguing issue for academics. The problem was resolved using a class of numerical or stochastic methods. Lee and Matsoukas \cite{lee2000simultaneous} employed a stochastic process, namely the constant-N Monte Carlo method, to solve the aggregation with a binary breakage equation. In 2002, Mahoney \et used the finite element method for aggregation, growth, and nucleation equations \cite{mahoney2002efficient}. Further,  number density and moments were computed with the help of the method of moments by Madras \et in \cite{madras2004reversible}. The model \eqref{breakage_eqn} and \eqref{agg_eqn} was also solved by implementing the finite volume method for several test cases in \cite{kumar2009comparison,kumar2013moment,kumar2015development,bourgade2008convergence}. Since, numerical schemes have some limitations and till date, there is no literature on semi-analytical schemes for coupled aggregation-breakage model, here we implement the AHPETM for solving the combined equation for two test cases. Moving further, the analytical solutions for the bivariate aggregation equation are available for limited cases \cite{fernandez2007exact,gelbard1980simulation,leyvraz2003scaling,lushnikov2004sol}. Several numerical methods, such as moving sectional \cite{kim1990simulation}, finite difference \cite{mantzaris2001numerical}, Monte Carlo \cite{matsoukas2009bicomponent}, sectional quadrature \cite{attarakih2010multivariate}, dual quadrature \cite{favero2012dual}, finite volume schemes \cite{kaur2017weighted,singh2016improved}, and many more \cite{vale2005solution,kumar1995general,chaudhury2013extended,kumar2008cell}, are considered to solve the equation. Therefore, our third aim here is to fill this gap of series solution for finding the approximate results for bivariate aggregation PBE.

\par The article is organized as follows: Section 2 discusses a brief outline of the Elzaki transformation. In Section 3, the general methodology of HPM and AHPETM are presented. In Section 4, AHPETM is developed for aforementioned population balance equations. Further, Section 5 gives a detailed convergence analysis of the proposed iterative scheme. In Section 6,  the developed formulations are adopted to demonstrate solutions for several kernels and the supremacy of the scheme over HPM, ADM, HAM, and ODM solutions are shown by means of numerical simulations.

		\section{Elzaki Transformation and Its Properties}
		
Tarig Elzaki developed the Elzaki transformation in 2011 \cite{elzaki2011new,elzaki2011application,elzaki2011connections,elzaki2011elzaki}, which is the modification of the general Laplace and Sumudu transformations  to solve the differential equation in the time domain. In \cite{elzaki2011new,elzaki2011application}, authors show the efficiency and accuracy of the Elzaki transformation on a large class of differential and integral equations. To understand the definition of the transformation, consider a set
$$A=\left\{f(t):\exists M, k_1,k_2>0,|f(t)|<M e^{\frac{|t|}{k_j}}, \text{ if } t \in (-1)^j\times [0,\infty)\right\}$$
then the Elzaki transformation is defined as 
$$E[f(t)]= T[v]= v \int_{0}^{\infty} f(t) e^{-\frac{t}{v}}dt, \quad t >0,$$ 
and the inverse of Elzaki transformation \cite{elzaki2012new} is defined as 
$$E^{-1}[T[v]]=\frac{1}{2 \pi i} \int_{0}^{\infty} e^{tv}T\left[\frac{1}{v}\right]v dv.$$
Some of the Elzaki transformation for standard functions are listed in TABLE \ref{p}.
\begin{table}[h]
\begin{center}
\caption{Properties of Elzaki transformation}
\label{p}
\begin{tabular}{c  c} \toprule
    {$f(t)$} & {$E[f(t)]$} \\\midrule
    1 & {$v^2$}\\
    &\\
    {$t^n$} & {$n! v^{n+2}$}\\
        &\\
    {$e^{at}$} & {$\dfrac{v^2}{1-a v}$}\\
        &\\
    {$E[f(t)+g(t)]$} & {$E[f(t)]+E[g(t)]$}\\
        &\\
    {$E[f^n(t)]$} & {$\quad \quad \frac{T[v]}{v^n}-\sum_{k=0}^{n-1} v^{2-n+k}f^k(0), \quad n \geq 1 $}\\
    
      \bottomrule
\end{tabular}
\end{center}
\end{table}
		\section{Methodology}
		In this section, we review the basics of HPM and AHPETM for solving general differential equations. Then the schemes are applied to solve multi-dimensional coagulation and coupled coagulation-fragmentation equations.
\subsection{Review of HPM}\label{sec31}
Let us consider the general differential equation
\begin{align}\label{gde}
D(c)-h(x)=0, \quad x \in \Omega
\end{align}
with the boundary conditions
\begin{align}
B\left(c,\frac{\pd c}{ \pd n}\right)=0, r \in \pd \Omega,
\end{align}
where $D$ and $B$ are the differential and boundary operators, respectively.  One can usually decompose the differential operator into linear ($L$) and non-linear ($N$) operators, implying that equation \eqref{gde} becomes
\begin{align}
L(c)+N(c)-h(x)=0.
\end{align}  
Now, according to HPM, a homotopy $H:\Omega\times [0,1]\rightarrow \mathbb{R}$ is constructed that satisfies
\begin{align}\label{homotopy}
H[v(x,p)]= (1-p)[L[v(r,p)]-L[(c_0)]]+p[D[v(r,p)]-h(x)]=0,
\end{align}
where $c_0$ is the initial guess for the equation \eqref{gde} and $p$ is the embedding parameter that increases monotonically from 0 to 1.
According to the HPM, we can write the solution of the equation \eqref{gde} in the form of series as 
\begin{align}\label{series_solution}
v= \sum_{k=0}^{\infty} p^k v_k= v_0+p v_1+p^2 v_2+\cdots.
\end{align}
Substituting equation \eqref{series_solution} in \eqref{homotopy} and letting $p \rightarrow 1$, the solution is obtained as follows
\begin{align}
c=\lim\limits_{p \rightarrow 1} v= \sum_{k=0}^{\infty} v_k.
\end{align}
\subsection{Accelerated Homotopy Perturbation Elzaki Transformation Method (AHPETM)}\label{sec32}
Consider a non-linear differential equation 
\begin{align}\label{gde1}
\frac{\pd^n c}{\pd t^n}+L[c(x,t)]+N[c(x,t)]= b(x)
\end{align}
with the initial conditions $ c^i(x,0)=g_i(x), \quad i=0,1,2,\cdots,n-1$, where $c^i(x,t)$ denotes the $i^{th}$ order derivative of $c(x,t)$ with respect to $t$.
Taking Elzaki transformation  and using its properties on equation \eqref{gde1} provide, by following \cite{jasrotia2022accelerated},
\begin{align}\label{transformed}
E[c(x,t)]=\sum_{k=0}^{n-1}v^{k+2}c^k(x,0)+ v^n E [b(x)-L[c(x,t)]-N[c(x,t)]].
\end{align}
Now, applying the homotopy perturbation method to the equation \eqref{transformed}, we get
\begin{align}\label{2}
(1-p) (E[c(x,t)]-E[c(x,0)])+p\left(E[c(x,t)]-\sum_{k=0}^{n-1}v^{k+2}c^k(x,0)- v^n E [b(x)-L[c(x,t)]-N[c(x,t)]]\right)=0.
\end{align}
Let the unknown function $c(x,t)$ and non-linear operator $N[c(x,t)]$ can be written in series form as
\begin{align}\label{series_solution1}
c(x,t)&=\sum_{n=0}^{\infty} v_n p^n
\end{align}
and 
\begin{align}\label{non_linear}
N[c(x,t)]=\sum_{n=0}^{\infty} H_n p^n
\end{align}
where $H_n$ represents the accelerated He's polynomial with
\begin{align}
H_n(x,t)=N(\sum_{i=0}^{n} v_i)-\sum_{i=0}^{n-1}H_i, \text{ for } n\geq 1 \text{ and } H_0= N(v_0).
\end{align}
Substituting the values of $c(x,t)$ and $N[c(x,t)]$ from the equations \eqref{series_solution1} and \eqref{non_linear} into equation \eqref{2} give
\begin{align*}
E[\sum_{n=0}^{\infty} v_n p^n]= \sum_{k=0}^{n-1}v^{k+2}c^k(x,0)+p\left\{v^nE\left[g(x)-L[\sum_{n=0}^{\infty}v_n p^k]+\sum_{n=0}^{\infty}H_n p^n\right]\right\}.
\end{align*}
Applying inverse Elzaki transformation and comparing the coefficients of powers of $p$, the components of series solution, i.e., $v_i's$ are given in TABLE \ref{t2} 
\begin{table}[h]
\begin{center}
\caption{Components of series solution}
\label{t2}
\begin{tabular}{c|  c} \toprule
    {$v_0$} & {$c(x,0)$} \\
        &\\
    {$v_1$} & {$\quad \sum_{k=1}^{n-1}\frac{t^k}{k!}c^k(x,0)+E^{-1}\{v^nE[b(x)-L[v_0]+H_0]\}$}\\
        &\\
    {$v_2$} & {$-E^{-1}\{v^nE[L[v_1]+H_1]\}$}\\
    
    {$\vdots$} & {$\vdots$}\\

    {$v_n$} & {$-E^{-1}\{v^nE[L[v_{n-1}]+H_{n-1}]\} $}\\
    
      \bottomrule
\end{tabular}
\end{center}
\end{table}  
and hence the solution of the equation \eqref{gde1} is obtained by taking $p\rightarrow 1$ in the equation \eqref{series_solution1}.
		\section{Development of Mathematical formation using AHPETM}
		In this AHPETM is extended to solve the Smoluchowski's coagulation, pure fragmentation, coupled coagulation-fragmentation and bivariate coagulation equations.
\subsection{Smoluchowski's Coagulation Equation (SCE)}
Consider the non-linear aggregation equation \eqref{agg_eqn} with initial condition $u(x,0)=u_0(x)$. Applying Elzaki transformation, an integral form is obtained as
\begin{align}\label{Eagg}
E[u(x,t)]= v^2 u(x,0)+ v E\left[\frac{1}{2}\int_{0}^{x}K(x-y,y) u(x-y,t) u(y,t)dy-\int_{0}^{\infty} K(x,y) u(x,t) u(y,t)dy\right].
\end{align}
In order to apply the scheme, compare equation \eqref{Eagg} with the transformed equation \eqref{transformed}, which provides $L[u]=0, \quad b(x)=0 $ and 
\begin{align}
 N[u]= -\frac{1}{2}\int_{0}^{x}K(x-y,y) u(x-y,t) u(y,t)dy+\int_{0}^{\infty} K(x,y) u(x,t) u(y,t)dy.
\end{align}
Now, applying the HPM on equation \eqref{Eagg} as defined in equation \eqref{2}, we get
\begin{align}\label{3}
(1-p) (E[c(x,t)]-E[c(x,0)])+p\bigg(E[c(x,t)]-v^2c(x,0)- v E \bigg[ \frac{1}{2}\int_{0}^{x}K(x-y,y) u(x-y,t) u(y,t)dy\nonumber\\-\int_{0}^{\infty} K(x,y) u(x,t) u(y,t)dy\bigg]\bigg)=0.
\end{align}
According to the methodology defined in Section \ref{sec32},  $c(x,t)=\sum_{n=0}^{\infty} v_n p^n $ and the non-linear operator  $N[u]=\sum_{n=0}^{\infty} H_np^n$, where $H_n$ for SCE is given by 
\begin{align}\label{H}
H_n=\frac{1}{2} \int_{0}^{x} K(x-y,y) \sum_{i=0}^{n} v_i(x-y,t) \sum_{i=0}^{n} v_i(y,t)dy-\int_{0}^{\infty} K(x,y) \sum_{i=0}^{n} v_i(x,t) \sum_{i=0}^{n} v_i(y,t)dy -\sum_{i=0}^{n-1}H_i, n \geq 1
\end{align}
with $H_0= N[v_0].$
Using the above defined decomposition in equation \eqref{3} and comparing the powers of $p$, the $n^{th}$ component of the series solution is
\begin{align}\label{agg_iterations}
v_{n+1}(x,t)= E^{-1}\bigg\{vE\bigg(\frac{1}{2}\int_{0}^{x}K(x-y,y)\sum_{i=0}^{n}v_i(x-y,t)\sum_{i=0}^{n}v_i(y,t)dy\nonumber\\-\int_{0}^{\infty}K(x,y)\sum_{i=0}^{n}v_i(x,t)\sum_{i=0}^{n}v_i(y,t)dy\bigg)-\sum_{i=0}^{n}H_i\bigg\}
\end{align}
where $v_0(x,t)=u(x,0)$ and hence, the $n$ term truncated series solution is calculated by 
\begin{align}\label{agg_trn}
\Psi_n^{CE}(x,t):= \sum_{j=0}^{n}v_j(x,t).
\end{align}
\subsection{Fragmentation Equation (FE)}
Considering the pure fragmentation equation \eqref{breakage_eqn} and applying Elzaki transformation, the following integral operator form is achieved
\begin{align}\label{breakage_transformed}
E[u(x,t)]=v^2 u(x,0)+E\left(\int_{x}^{\infty} B(x,y) S(y) u(y,t)dy-S(x)u(x,t)\right).
\end{align}
Next, equation \eqref{breakage_transformed} is compared with the equation \eqref{transformed} for the implementation of AHPETM. It is observed that for the case of pure breakage equation $N[u(x,t)]=b(x)=0$ and
$$L[u(x,t)]=-\int_{x}^{\infty} B(x,y) S(y) u(y,t)dy+S(x)u(x,t).$$ 
By following the steps discussed in the previous Section \ref{sec32}, a homotopy is generated as follows
\begin{align}\label{4}
(1-p)\{E[u(x,t)]-E[u(x,0)]\}+p\left(E[u(x,t)]-v^2u(x,0)-vE\left[\int_{x}^{\infty} B(x,y)S(y)u(y,t)dy-S(x)u(x,t)\right]\right).
\end{align}
According to the proposed method, AHPETM introduces the solution of unknown function $u(x,t)$ in the form of infinite series as $u(x,t)=\sum_{j=0}^{\infty}v_j(x,t)$. Substituting this into equation \eqref{4} and comparing the coefficients of the power of $p$, provide the iterations for the solution as follows
\begin{align}\label{breakage_iterations}
v_{n+1}(x,t)= E^{-1}\left\{vE\left[\int_{x}^{\infty}B(x,y)S(y)v_n(x,t)dy-S(x)v_n(x,t)\right]\right\}
\end{align}
where $v_0(x,t)=u(x,0)$ and the $n$ term truncated solution will be provided as 
\begin{align}\label{FE_trn}
\Psi_n^{FE}(x,t):= \sum_{j=0}^{n}v_j(x,t).
\end{align}
\subsection{Coupled Coagulation-fragmentation Equation (CCFE)} The CCFE is governed by 
\begin{align}\label{agg_break_eqn}
\frac{\pd u(x,t)}{\pd t}= \frac{1}{2} \int_{0}^{x} K(x-y,y) u(x-y,t) u(y,t) dy-\int_{0}^{\infty} K(x,y) u(x,t) u(y,t) dy\nonumber\\+\int_{x}^{\infty}B(x,y)S(y)u(y,t)dy-S(x)u(x,t).
\end{align}
Applying Elzaki transformation on both sides leads to
\begin{align}\label{Eaggbrk}
E[u(x,t)]= v^2 u(x,0)+ v E\bigg[\frac{1}{2}\int_{0}^{x}K(x-y,y) u(x-y,t) u(y,t)dy-\int_{0}^{\infty} K(x,y) u(x,t) u(y,t)dy\nonumber\\+\int_{x}^{\infty} B(x,y) S(y) u(y,t)dy-S(x)u(x,t)\bigg].
\end{align}
For the implementation of AHPETM, expression \eqref{Eaggbrk} is compared with \eqref{transformed} and the following observations are made
$$b(x)=0, \quad L[u]=-\int_{x}^{\infty} B(x,y) S(y) u(y,t)dy+S(x)u(x,t),$$ and 
\begin{align*}
 N[u]= -\frac{1}{2}\int_{0}^{x}K(x-y,y) u(x-y,t) u(y,t)+\int_{0}^{\infty} K(x,y) u(x,t) u(y,t).
\end{align*}
Following the procedure defined in Section \ref{sec32}, the iterations to solve the equation \eqref{agg_break_eqn} are as follows
\begin{align}\label{agg_break_iterations}
v_{n+1}(x,t)=& E^{-1}\bigg\{vE\bigg(\frac{1}{2}\int_{0}^{x}K(x-y,y)\sum_{i=0}^{n}v_i(x-y,t)\sum_{i=0}^{n}v_i(y,t)dy\nonumber\\-\int_{0}^{\infty}K(x,y)&\sum_{i=0}^{n}v_i(x,t)\sum_{i=0}^{n}v_i(y,t)dy-\sum_{i=0}^{n}H_i+\int_{x}^{\infty}B(x,y)S(y)v_n(x,t)dy-S(x)v_n(x,t)\bigg)\bigg\},
\end{align}
where $v_0(x,t)=u(x,0).$ Let us denote the $n$ term approximated series solution for CCFE as 
\begin{align}\label{agg_break_sol}
\Psi_n^{CCFE}(x,t):= \sum_{j=0}^{n}v_j(x,t).
\end{align}
\subsection{Bivariate Smoluchowski's Coagulation Equation (BSCE)}
Consider 2D aggregation equation \eqref{2DPBE} with initial condition $u(x,y,0)=u_0(x,y)$ and applying Elzaki transformation, leads to form as
\begin{align}\label{E2dagg}
E[u(x,y,t)]=& v^2 u(x,y,0)+ v E\bigg[\frac{1}{2}\int_{0}^{x}\int_{0}^{y}K(x-x',y-y',x',y') u(x-x',y-y',t) u(x',y',t)dy'dx'\nonumber\\&-\int_{0}^{\infty}\int_{0}^{\infty} K(x,x',y,y') u(x,y,t) u(x',y',t)dy'dx'\bigg].
\end{align}
In order to apply the AHPETM, equation \eqref{E2dagg} is compared with the transformed equation \eqref{transformed}, implying that $L[u]=0, \quad b(x)=0 $ and 
\begin{align}
 N[u]= & -\frac{1}{2} \int_{0}^{x}\int_{0}^{y}K(x-x',y-y',x',y')u(x-x',y-y',t)u(x',y',t)dy'dx' \nonumber \\
 		+&\int_{0}^{\infty}\int_{0}^{\infty} K(x,x',y,y')u(x,y,t)u(x',y',t)dy'dx'.
\end{align}
Thanks to equation \eqref{2}, applying the HPM on equation \eqref{E2dagg} enables us to have
\begin{align}\label{5}
&(1-p) (E[c(x,t)]-E[c(x,0)])+p\bigg(E[c(x,y,t)]-v^2c(x,y,0)- v E \bigg[ \frac{1}{2}\int_{0}^{x}\int_{0}^{y}K(x-x',y-y',x',y')\nonumber\\& u(x-x',y-y',t) u(x',y',t)dy'dx'-\int_{0}^{\infty}\int_{0}^{\infty}K(x,x',y,y')u(x,y,t)u(x',y',t)dy'dx'\bigg]\bigg)=0.
\end{align}
Again, following the idea of Section \ref{sec32},  $u(x,y,t)=\sum_{n=0}^{\infty} v_n p^n $ and non-linear operator $N[u]=\sum_{n=0}^{\infty} H_np^n$ where $H_n$ is being given by 
\begin{align}
H_n=&\frac{1}{2} \int_{0}^{x} K(x-x',y-y',x',y') \sum_{i=0}^{n} v_i(x-x',y-y',t) \sum_{i=0}^{n} v_i(x',y',t)dy'dx'\nonumber\\&-\int_{0}^{\infty} K(x,x',y,y') \sum_{i=0}^{n} v_i(x,y,t) \sum_{i=0}^{n} v_i(x',y',t)dy'dx' -\sum_{i=0}^{n-1}H_i \text{ with } H_0=N[v_0].
\end{align}
Using the above defined decomposition in equation \eqref{5} and comparing the powers of $p$, we get the $n^{th}$ component of the series solution as follows
\begin{align}\label{2dagg_iterations}
v_{n+1}(x,y,t)=& E^{-1}\bigg\{vE\bigg(\frac{1}{2}\int_{0}^{x}\int_{0}^{y}K(x-x',y-y',x',y')\sum_{i=0}^{n}v_i(x-x',y-y',t)\sum_{i=0}^{n}v_i(x',y',t)dy'dx'\nonumber\\&-\int_{0}^{\infty}\int_{0}^{\infty}K(x,x',y,y')\sum_{i=0}^{n}v_i(x,y,t)\sum_{i=0}^{n}v_i(x',y',t)dy'dx'\bigg)\bigg\}
\end{align}
where $v_0(x,y,t)=u(x,y,0).$ Let us denote the $n$ term truncated solution by 
\begin{align}\label{2dagg_sol}
\Psi_n^{BSCE}(x,y,t):= \sum_{j=0}^{n}v_j(x,y,t).
\end{align}
		\section{Convergence Analysis}
		\subsection{Smoluchowski's Coagulation Equation}
Consider a Banach space $\mathbb{X}=\mathbb{C}( [0,T]:\mathbb{L}^1[0,\infty),\|.\|)$ over the norm defined as 
$$\|u\|=\sup_{s \in [0,t_0]}\int_{0}^{\infty} |u(x,s)|dx< \infty.$$
Let us use equation \eqref{Eagg} in the operator form as
\begin{align*}
u(x,t)=\tilde{\mathcal{N}}[u]
\end{align*}
where
\begin{align}\label{tN}
\tilde{\mathcal{N}}[u]=u(x,0)+E^{-1}\{vE[N[u]]\}
\end{align}
and $N[u]$ is given by
\begin{align*}
N[u]=\frac{1}{2}\int_{0}^{x} K(x-y,y) u(x-y,t) u(y,t)dy-\int_{0}^{\infty} K(x,y) u(x,t) u(y,t)dy.
\end{align*}
\begin{theorem}\label{tm0}
   Let us consider the coagulation equation \eqref{agg_eqn} with kernel $K(x,y)=1$ for all $x,y \in (0,\infty)$. If $v_i^s $ are the components of the series solution computed using \eqref{agg_iterations} and $\Psi_n^{CE}$ being the $n$ term truncated solution provided in equation \eqref{agg_trn}, then  $\Psi_n^{CE}$ converges to the exact solution $u$ with the error bound $$\|u-\Psi_m^{CE}\|\leq \dfrac{\Delta^m }{1-\Delta}\|v_1\|$$ where  $\Delta= t_0^2 e^{2 t_0 L}(\|u_0\|+2 t_0 L^2 + 2 t_0 L) < 1 \text{ and }  L= \|u_0\| (T+1)$.
\end{theorem}
\begin{proof}
Two separate phases complete the theorem's proof. The contractive nature of the non-linear operator $\tilde{\mathcal{N}}$ is initially demonstrated. Then convergence of the truncated solution towards the exact one is established.\\
\textbf{Step 1:} As presented in \cite{singh2015adomian}, equation \eqref{tN} can be written in the equivalent form as
\begin{align*}
\frac{\pd }{\pd t}[u(x,t)\exp[H[x,t,u]]]=\frac{1}{2} \exp[H[x,t,u]]\int_{0}^{x} K(x-y,y)u(x-y,t)u(y,t)dy
\end{align*}
where $H[x,t,u]=\int_{0}^{t}\int_{0}^{\infty}K(x,y)u(y,s)dyds$. Thus the equivalent operator $\tilde{N}$ is given by
\begin{align*}
\tilde{N}[u]=u(x,0)\exp[-H(x,t,u)]+\frac{1}{2}\int_{0}^{t} \exp[H(x,s,u)-H(x,t,u)]\int_{0}^{\infty} K(x,y) u(x-y,s) u(y,s)dyds.
\end{align*}
Since $\tilde{N}$ is contractive (Singh \et established in  \cite{singh2015adomian}) and equivalent to $N[u]$, the non-linear operator $N[u]$ is also contractive, i.e.,
\begin{align}\label{cont_tN}
\|Nu-Nu^*\| \leq \delta \|u-u^*\| 
\end{align}
 where $\delta:= t_0e^{2 t_0 L}(\|u_0\|+ 2t_0 L^2+2 t_0 L) <1$  (for suitably chosen $t_0$) and $L= \|u_0\| (T+1)$.\\
Now, using the definition and basic properties of Elzaki and Laplace transformations as well as employing \eqref{cont_tN}, we get
\begin{align*}
\|\tilde{\mathcal{N}}u-\tilde{\mathcal{N}}u^*\|&= \|E^{-1}\{vE(N(u))\}-E^{-1}\{vE(N(u^*))\}\|\\
&= \left\|\frac{1}{2 \pi} \int_{0}^{\infty} \left(\frac{1}{v^2}\int_{0}^{\infty} (Nu-Nu^*)e^{-tv}dt\right)e^{tv}v dv\right\|\\ 
&\leq  \frac{1}{2 \pi} \int_{0}^{\infty} \left(\frac{1}{v}\int_{0}^{\infty} \delta\|u-u^*\|e^{-tv}dt\right)e^{tv} dv\\
&= \frac{1}{2 \pi}\int_{0}^{\infty} \frac{1}{v} \mathcal{L}(\delta \|u-u^*\|)e^{tv}dv\\
&= \mathcal{L}^{-1}\left\{\frac{1}{v^2}\mathcal{L}(\delta\|u-u^*\|)\right\}\leq\delta t_0\|u-u^*\| \text{ for a suitable } t_0.
\end{align*}
\textbf{Step 2:} Now, in this phase, an $n$ term truncated solution is computed using the iterations defined in \eqref{agg_iterations} and then error is estimated. Given that,
\begin{align*}
\Psi_n^{CE}=&\sum_{i=0}^{n}v_i(x,t)\\
=& u(x,0)+E^{-1}\{vE(N(u_0))\}+E^{-1}\{vE(N(u_0+u_1)-H_0)\}+\cdots+E^{-1}\{vE(N(\sum_{j=0}^{n-1}u_j)-\sum_{j=0}^{n-2}H_i)\}\\
=& u(x,0)+ E^{-1}\{v E(N[v_0]+N[v_0+v_1]+\cdots+N[v_0+v_1+\cdots+v_{n-1}]-\\&(H_0+(H_0+H_1)+\cdots+(H_0+H_1+\cdots+H_{n-2})))\}\\
=& u(x,0)+E^{-1}\{vE(N(\Psi_{n-1}^{CE}))\}= \mathcal{\tilde{N}}[\Psi_{n-1}^{CE}].
\end{align*}
Using the contractive mapping of $\mathcal{\tilde{N}}$ leads to
$$\|\Psi_{n+1}^{CE}-\Psi_n^{CE}\| \leq \Delta \|\Psi_n^{CE}-\Psi_{n-1}^{CE}\|.$$
and thus, we have
\begin{align*}
\|\Psi_{n+1}^{CE}-\Psi_n^{CE}\| & \leq \Delta\|\Psi_n^{CE}-\Psi_{n-1}^{CE}\| \leq \Delta^n\|\Psi_{1}^{CE}-\Psi_0^{CE}\|.
\end{align*}
Using the triangle inequality for all $n,m \in \mathbb{N}$ with $n>m$, we have 
\begin{align*}
\|\Psi_n^{CE}-\Psi_m^{CE}\| &\leq \|\Psi_n^{CE}-\Psi_{n-1}^{CE}\|+\|\Psi_{n-1}^{CE}-\Psi_{n-2}^{CE}\|+\cdots+\|\Psi_{m+1}^{CE}-\Psi_m^{CE}\|\\ & \leq (\Delta^{n-1}+\Delta^{n-2}+\cdots+\Delta^{m})\|\Psi_1^{CE}-\Psi_0^{CE}\|\\ &= \dfrac{\Delta^m (1-\Delta^{n-m})}{1-\Delta}\|u_1\| \leq \dfrac{\Delta^m}{1-\Delta}\|u_1\|,
\end{align*}
which converges to zero as $m\rightarrow \infty$, implies that there exists a $\Psi$ such that $\lim\limits_{n\rightarrow \infty}\Psi_n^{CE}=\Psi.$ Therefore,
\begin{align*}
u(x,t)=\sum_{i=0}^{\infty} v_i=\lim\limits_{n\rightarrow \infty}\Psi_n^{CE}=\Psi,
\end{align*}
which is the exact solution of the coagulation equation \eqref{agg_eqn}. The theoretical error is obtained by fixing $m$ and letting $n \rightarrow \infty$ in the above formulation.
\end{proof}
\subsection{Pure Breakage Equation}
Let $\mathbb{X}= \mathbb{C}([0,T]: \mathbb{L}^{1}[0,\infty),\|.\|])$ be a Banach space with the norm 
\begin{align}
\|u\|=\sup_{t \in [0,t_0]} \int_{0}^{\infty} e^{\lambda x} |u(x,t)| dx, \text{ where } \lambda >0.
\end{align}
Now, equation \eqref{breakage_eqn} can be rewritten in the operator form as 
$$u=\mathcal{\tilde{L}}[u]=u(x,0)+E^{-1}{vE(L[u])}$$ 
with $L[u]$ being the right-hand side of equation \eqref{breakage_eqn}.
\begin{theorem}
Let $\Psi_n^{FE}$ be the $n$ term truncated series solution of the fragmentation problem defined in equation \eqref{breakage_eqn}. Then $\Psi_n^{FE}$ converges to the exact solution and provides the error estimates
\begin{align}\label{ee}
\|u-\Psi_m^{FE}\|\leq \dfrac{\vartheta^m}{1-\vartheta}\|v_1\|,
\end{align}
if the following conditions hold
\begin{itemize}
\item[B1.] $B(x,y)=c\dfrac{x^{r-1}}{y^r}$ where $r=1,2,3,\cdots$ and $c$ is a positive constant satisfying $\int_{0}^{y} x B(x,y)dx=y,$
\item[B2.] $S(x)\leq x^k, \text{ where } k=1,2,3,\cdots,$ 
\item[B3.]$\lambda$ is chosen such that $e^{\lambda y}-1 <1,$
\item[B4.] $\vartheta:= \dfrac{k! (t_0)^2}{ \lambda^{k+1}}<1$ for suitable  $t_0$.
\end{itemize} 
\end{theorem}
\begin{proof}
Let us begin with the proof that the operator $\mathcal{\tilde{L}}$ is contractive. In order to do so, we use the fact that the operator $L[u]$ is a contractive operator under the assumptions mentioned in B1-B3 i.e.,, $\|L[u]-L[u^*]\|\leq \rho \|u-u^*\|$ where $\rho= \dfrac{k! t_0}{\lambda^{k+1}}< 1$ by following (\cite{singh2015adomian} Theorem 2.1). Now, thanks to Elzaki and Laplace transformations, one can write
\begin{align*}
\|\mathcal{L}[u]-\mathcal{L}[u^*]\|&= \|E^{-1}\{vE(L[u])\}-E^{-1}\{vE(L[u^*])\}\|\\
&=\left\|\frac{1}{2 \pi} \int_{0}^{\infty} \left(\frac{1}{v^2}\int_{0}^{\infty} (Lu-Lu^*)e^{-tv}dt\right)e^{tv}v dv\right\|\\ 
&\leq \frac{1}{2 \pi} \int_{0}^{\infty} \left(\frac{1}{v}\int_{0}^{\infty} \rho\|u-u^*\|e^{-tv}dt\right)e^{tv} dv\\
&= \frac{1}{2 \pi}\int_{0}^{\infty} \frac{1}{v} \mathcal{L}(\rho \|u-u^*\|)e^{tv}dv\\
&= \mathcal{L}^{-1}\left\{\frac{1}{v}\mathcal{L}(\rho\|u-u^*\|)\right\}\leq \vartheta\|u-u^*\| \text{ where } \vartheta= \rho t_0.
\end{align*}
 We proceed further to obtain the estimate \eqref{ee}. By using the iteration formula \eqref{breakage_iterations}, the $n$-term truncated solution is computed as
 \begin{align*}
 \Psi_n^{FE}=& E^{-1}\{vE[v_0]\}+E^{-1}\{vE[v_1]\}+\cdots+E^{-1}\{vE[v_{n-1}]\}\\
 =& E^{-1}\{vE[v_0+v_1+\cdots+v_{n-1}]\}= E^{-1}\{vE[\Psi_{n-1}^{FE}]\}.
 \end{align*}
Therefore, we have
 \begin{align*}
 \|\Psi_{n+1}^{FE}-\Psi_n^{FE}\| \leq \vartheta\|\Psi_n^{FE}-\Psi_{n-1}^{FE}\| \leq \vartheta^n \|\Psi_{1}^{FE}-\Psi_{0}^{FE}\|.
 \end{align*}
 The above can be used to establish the following, for all $n,m \in \mathbb{N}$ with $n>m$,
\begin{align*}
\|\Psi_n^{FE}-\Psi_m^{FE}\| &\leq \|\Psi_n^{FE}-\Psi_{n-1}^{FE}\|+\|\Psi_{n-1}^{FE}-\Psi_{n-2}^{FE}\|+\cdots+\|\Psi_{m+1}^{FE}-\Psi_m^{FE}\|\\& \leq (\vartheta^{n-1}+\vartheta^{n-2}+\cdots+\vartheta^{m})\|\Psi_1^{FE}-\Psi_0^{FE}\|\\ &= \dfrac{\vartheta^m (1-\vartheta^{n-m})}{1-\vartheta}\|v_1\| \leq \dfrac{\vartheta^m}{1-\vartheta}\|v_1\|.
\end{align*}
Thanks for hypothesis B4, the above tend to zero as $m\rightarrow \infty$ which means that there exists a $\Psi$ such that $\lim\limits_{n\rightarrow \infty}\Psi_n^{FE}=\Psi$. Thus, we obtain the exact solution of the breakage equation \eqref{breakage_eqn} as
\begin{align*}
u(x,t)=\sum_{i=0}^{\infty} v_i=\lim\limits_{n\rightarrow \infty}\Psi_n^{FE}=\Psi.
\end{align*}
\end{proof}
\subsection{Bivariate Smoluchowski's Coagulation Equation}
Consider a Banach space $X= \mathbb{C}([0,T]:L^1[0,\infty)\times L^1[0,\infty),\|.\|)$ with the enduced norm 
$$\|u\|=\sup_{s \in [0,t_0]} \int_{0}^{\infty}\int_{0}^{\infty}|u(x,y,s)|dx dy<\infty.$$
To demonstrate the convergence analysis, let us write the operator form of the equation \eqref{E2dagg} as
\begin{align}
u= \tilde{Q}[u],
\end{align}
where $\tilde{Q}$ is given by
\begin{align}\label{tQ}
\tilde{Q}[u]= u_0(x,y)+E^{-1}[v E[Q[u]]]
\end{align}
and 
\begin{align*}
Q[u]=&-\frac{1}{2} \int_{0}^{x}\int_{0}^{y}K(x-x',y-y',x',y')u(x-x',y-y',t)u(x',y',t)dy'dx'\nonumber\\&+\int_{0}^{\infty}\int_{0}^{\infty} K(x,x',y,y')u(x,y,t)u(x',y',t)dy'dx'.
\end{align*}
The iterative scheme's convergence concept is splitted into two components, firstly, we establish that the operator $\tilde{Q}$ is contractive (Theorem 3) and then proceed further to discuss the worst case upper bound for error (Theorem 4) below. To show the operator $\tilde{Q}$ contractive, initially we prove that $Q$ is contractive. To do so, an equivalent form of the equation \eqref{2DPBE} is taken as
\begin{align}
  \frac{\pd }{\pd t}[u(x,y,t)\exp[R(x,y,t,c)]]=\frac{1}{2} \exp[R(x,y,t,c)]\int_{0}^{x}\int_{0}^{y} K(x-x',x',y-y',y')u(x-x',y-y',t)u(x',y',t)dy'dx',
  \end{align} 
  where $R(x,y,t,c)=\int_{0}^{t}\int_{0}^{\infty}\int_{0}^{\infty}K(x,x',y,y')u(x',y',t)dx'dy'dt$.
  Thus the equivalent operator $\mathcal{N}$ is given by
  \begin{align}\label{2dequivalent}
  \mathcal{N}[u]=u(x,y,0)\exp[-R(x,y,t,u)]+\frac{1}{2}\int_{0}^{t}\exp[R(x,y,s,u)-R(x,y,t,u)]\nonumber\\\int_{0}^{x}\int_{0}^{y} K(x-x',x',y-y',y') u(x-x'y-y',s) u(x',y',s)dy'dx'ds.
  \end{align}
  Since, $\mathcal{N}$ and $Q$ are equivalent, it is sufficient to show that $\mathcal{N}$ is contractive.
  \begin{theorem}\label{tm1}
 The operator $\tilde{Q}$, defined in equation \eqref{tQ} is contractive for all $u, u^* \in \mathbb{X}$ if the following conditions
 \begin{itemize}
 \item $K(x,x',y,y')=1 \quad \forall x,x',y,y' \in (0,\infty)$ and 
 \item $\delta = 2t_0^2 e^{2 t_0 L}(\|u\|+ 2 t_0 L^2 + 2 t_0 L)<1$ where $L=\|u_0\| (T+1)$ hold.
 \end{itemize}
  \end{theorem}
  \begin{proof}
  Consider $u,u^* \in \mathbb{X},$ then
  \begin{align*}
  \mathcal{N}[u]-\mathcal{N}[u^*]=& u(x,y,0)\exp[-R(x,y,t,u)]-u^*(x,y,0)\exp[-R(x,y,t,u^*)]\\  &+\frac{1}{2}\int_{0}^{t}\exp[R(x,y,s,u)-R(x,y,t,u)]\int_{0}^{x}\int_{0}^{y}u(x-x',y-y',s) u(x',y',s)dy'dx'ds\\
     &-\frac{1}{2}\int_{0}^{t}\exp[R(x,y,s,c^*)-R(x,y,t,c^*)]\int_{0}^{x}\int_{0}^{y} u^*(x-x'y-y',s) u^*(x',y',s)dy'dx'ds.  
  \end{align*}
  Let us define an another operator $$H[x,y,s,t]= \exp\{R[x,y,s,u]-R[x,y,t,u]\}- \exp\{R[x,y,s,u^*]-R[x,y,t,u^*]\}.$$
  It can be easily proven that 
  $$|H[x,y,s,t]|\leq (t-s) \exp\{(t-s)B\}\|u-u^*\|\leq B_1 \|u-u^*\|, $$
  where $B_1= t e^{tB}$ and $B=\max\{\|u\|,\|u^*\|\}.$ 
  Further,
  \begin{align}\label{6}
   \mathcal{N}[u]-\mathcal{N}[u^*]=& u_0(x,y)H(x,y,0,t)+\frac{1}{2}\int_{0}^{t}H(x,y,s,t)\int_{0}^{x}\int_{0}^{y}u(x-x',y-y',s)u(x',y',s)dy'dx'ds \nonumber\\
       &+\frac{1}{2}\int_{0}^{t}\exp[R(x,y,s,u^*)-H(x,y,t,u^*)] \nonumber\\
       &\bigg[\int_{0}^{x}\int_{0}^{y} u^*(x-x'y-y',s)\{u(x',y',s)-u^*(x',y',s)\}dy'dx'\nonumber \\
       &+ \int_{0}^{x}\int_{0}^{y} u(x',y',s)\{u(x-x',y-y',s)-u^*(x-x',y-y',s)\}dx'dy'\bigg]ds.   
  \end{align}
 To show the non-linear operator $\mathcal{N}$ contractive, a set $D$ is defined such as $D= \{u\in \mathbb{X}: \|u\|\leq 2 L\}.$ Taking norm on both sides of \eqref{6} provides
\begin{align*}
 \| \mathcal{N}[u]-\mathcal{N}[u^*] \|&\leq B_1\|u-u^*\| \|u_0\|+B_1\|u-u^*\|\int_{0}^{t}\left[\frac{1}{2}\|u\|^2\right]ds +\int_{0}^{t}B_1\left[\frac{1}{2}(\|u\|+\|u^*\|)\|u-u^*\|\right]ds \nonumber\\
 & \leq B_1\left[\|u_0\|+\frac{1}{2}t\|u\|^2+\frac{1}{2}t(\|u\|+\|u^*\|)\right]\|u-u^*\| \nonumber \\
 & \leq t_0e^{2t_0L}\left[\|u_0\|+2t_0L^2+\frac{1}{2}t_0(2L+2L) \right]\|u-u^*\| \nonumber \\
 & = \Delta \|u-u^*\|,
\end{align*}
 for a suitable choice of $t_0$. So, the operator $\mathcal{N}$ is contractive if $\Delta =t_0e^{2t_0L}\left[\|u_0\|+2t_0L^2+2t_0L\right]<1$, hence $Q$ is contractive.
Now we are in position to demonstrate that $\tilde{\mathcal{Q}}$ is contractive. Consider,
 \begin{align*}
 \|\tilde{Q}u-\tilde{Q}u^*\|&= \|E^{-1}(vE\left[Qu\right])-E^{-1}(vE\left[Qu^*\right])\|\\
  &=\| \frac{1}{2 \pi} \int_{0}^{\infty} \left(\frac{1}{v^2}\int_{0}^{\infty} (Qv-Qv^*)e^{-vt}dt\right)e^{vt}v dv\|\\ 
  &\leq  \frac{1}{2 \pi} \int_{0}^{\infty} \left(\frac{1}{v}\int_{0}^{\infty} \|Qu-Qu^*\|e^{-vt}dt\right)e^{vt} dp\\
  & \leq \frac{1}{2 \pi} \int_{0}^{\infty} \left(\frac{1}{v}\int_{0}^{\infty} \delta\|u-u^*\|e^{-vt}dt\right)e^{vt} dv\\
  &= \frac{1}{2 \pi}\int_{0}^{\infty} \frac{1}{v} \mathcal{L}(\delta \|u-u^*\|)e^{vt}dv\\
  = \mathcal{L}^{-1}\left\{\frac{1}{v}\mathcal{L}(\delta\|u-u^*\|)\right\}=\delta t_0\|u-u^*\|= \Delta\|u-u^*\|.
 \end{align*}
 Which accomplishes the contractive nature of $\tilde{Q}$ by following the assumption.
  \end{proof}
  \begin{theorem}
  Assuming that the criteria of Theorem \eqref{tm1} hold and $v_i's$ are the elements of the series solution calculated by equation \eqref{2dagg_iterations}. Then the series solution converges to the exact solution with the error bound
  \begin{align*}
  \|u-\Psi_n^{FE}\|= \dfrac{\Delta^n}{1-\Delta}\|v_1\|,
  \end{align*}
  whenever $\Delta<1$ and $\|v_1\|<\infty.$
  \end{theorem}
  \begin{proof}
  The proof is similar to the Theorem \ref{tm0}, hence it is omitted here.
  \end{proof}
\begin{rem}
It is worth mentioning that the iterations and hence the finite term series solutions, computed using the HAM \cite{kaur2022approximate}, HPM  \cite{kaur2019analytical}, ADM \cite{singh2015adomian} and ODM \cite{kaushik2022novel} are identical to the iterations obtained using the AHPETM for the breakage equation which is linear. As a result, we have omitted the numerical implementations for pure breakage equation. So the  main focus of all the approaches is on approximating the non-linearity, which has no bearing on the linearity in the equations. 
\end{rem}
			\section{Numerical Results and Discussion}
				This section verifies numerically the effectiveness of the suggested approach for coagulation, combined fragmentation-coagulation, and bivariate aggregation equations.  Three physical test cases are considered and results for the number density and moments are compared with the precise solution as well as established and recently developed methods (ADM, HPM, HAM, ODM) for SCE. Due to the improved and significant results noticed in SCE, the numerical implementation is made for solving the coupled CFE and BSCE. Two test cases of CFE and one example of BSCE are taken into account to justify the effectiveness of our scheme. 
\subsection{Smoluchowski's Coagulation Equation}
\begin{example}
Consider the case of constant aggregation kernel $K(x,y)=1$ with the exponential initial data $u(x,0)=e^{-x}$ and for this the exact solution
$$u(x,t)=\frac{4}{(2+t)^2}e^{-\frac{2x}{2+t}},$$  is discussed in \cite{ranjbar2010numerical}.
\end{example}
Employing the equations \eqref{agg_iterations} and \eqref{agg_trn}, first three components of the series solutions are given as follows
$$v_0(x,t)= e^{-x},\quad
v_1(x,t)= \frac{1}{2} t e^{-x} (x-2),$$
\begin{align*}
v_2(x,t)=  t^3 \left(\frac{x^3}{144}-\frac{x^2}{12}+\frac{x}{4}-\frac{1}{6}\right) e^{-x}+t^2 \left(\frac{x^2}{8}-\frac{3 x}{4}+\frac{3}{4}\right) e^{-x},
\end{align*}
\begin{align*}
v_3(x,t)=& \frac{1}{40642560}t^3 e^{-x} \bigg(t^4 x^7+14 t^3 (7-4 t) x^6+588 (t-2) t^2 (2 t-3) x^5-2940 t (t (t (4 t-21)+36)-24) x^4\\& +11760 (5 (t-4) t ((t-3) t+6)+48) x^3-35280 (t (t (t (4 t-35)+120)-240)+192) x^2\\& +70560 (t (t (t (2 t-21)+90)-240)+288) x-10080 (t (t (t (4 t-49)+252)-840)+1344)\bigg).
\end{align*}
It is essential to mention here that the components $v_i$ are quite complicated and due to the complexity of the terms, it is hard to find a closed-form solution. Therefore, a three-term truncated solution is considered. However, thanks to "MATHEMATICA", one can compute the higher order terms using equation \eqref{agg_iterations}.\\
	\begin{figure}[h]
 		\centering
 	\subfigure[AHPETM $(n=3)$]{\includegraphics[width=0.45\textwidth,height=0.35\textwidth]{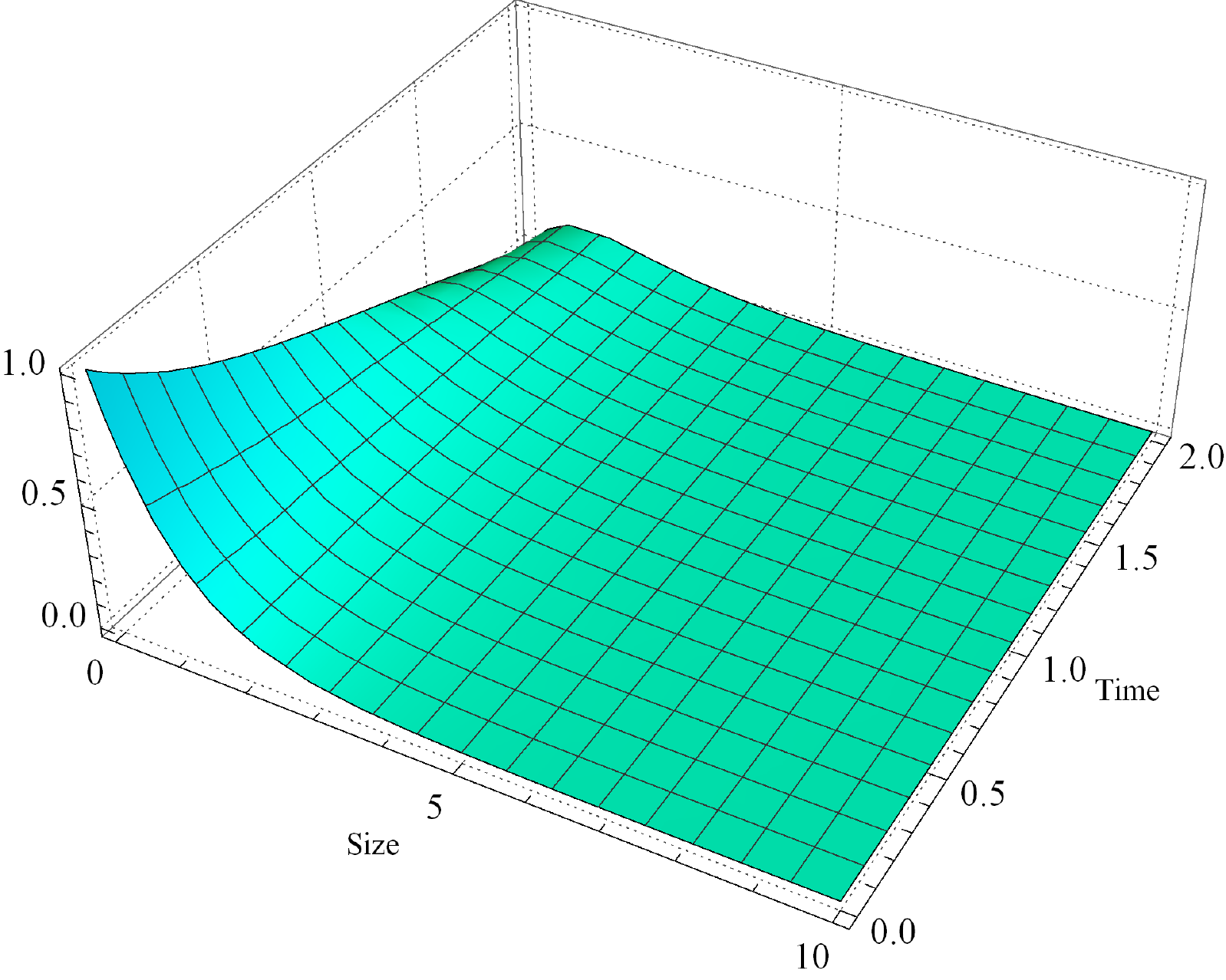}}
 	\subfigure[Exact]{\includegraphics[width=0.45\textwidth,height=0.35\textwidth]{{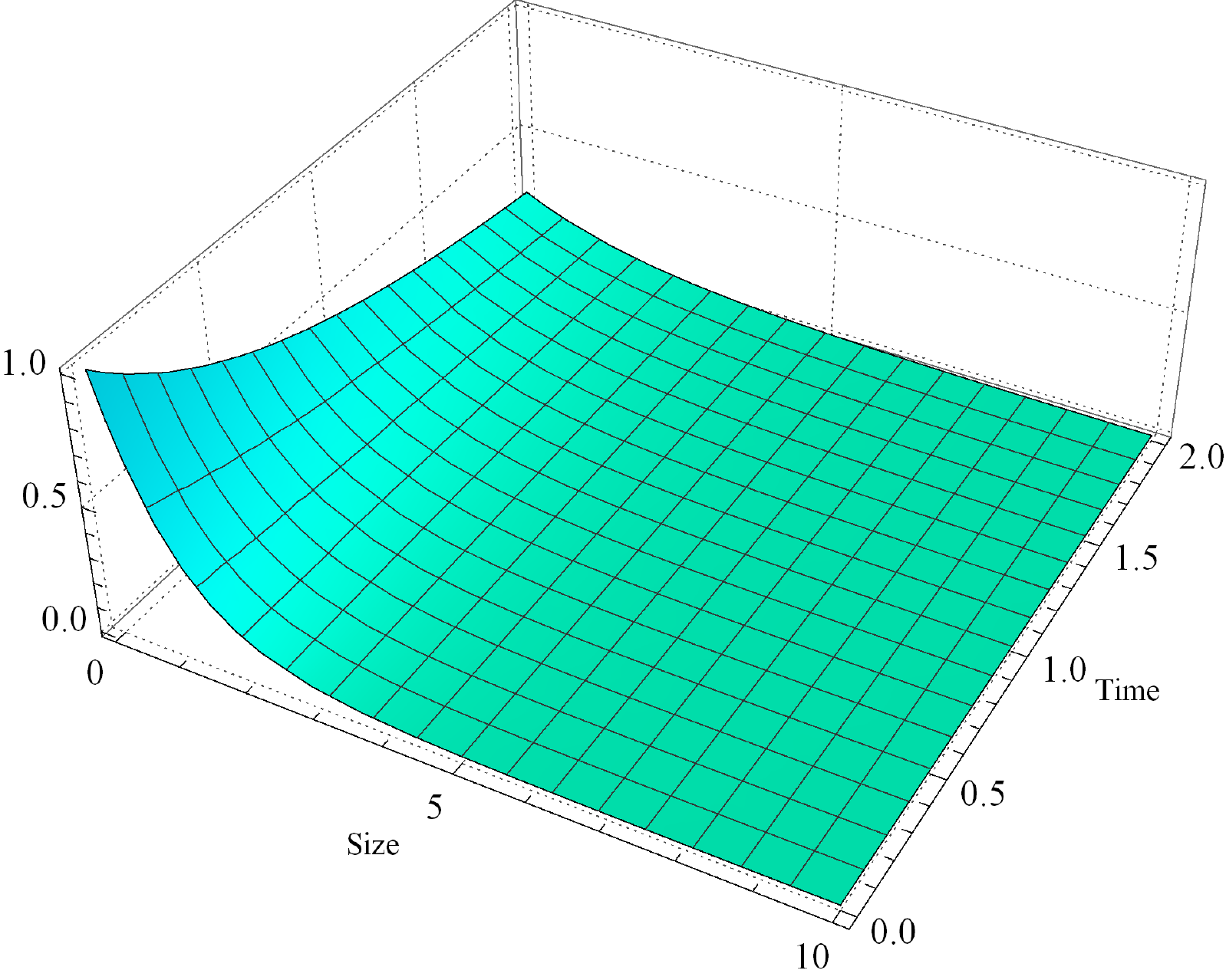}}}
\caption{Number density for AHPETM and exact solutions for Example 6.1}
\label{fig0}
\end{figure}
To see the accuracy of our proposed method, the approximated three-term and exact solutions are plotted in Figures \ref{fig0}(a) and \ref{fig0}(b). One can scrutinize that the AHPETM solution shows a remarkable agreement with the exact one.
	\begin{figure}[htb!]
 		\centering
 	\subfigure[AHPETM error $(n=3)$]{\includegraphics[width=0.45\textwidth,height=0.35\textwidth]{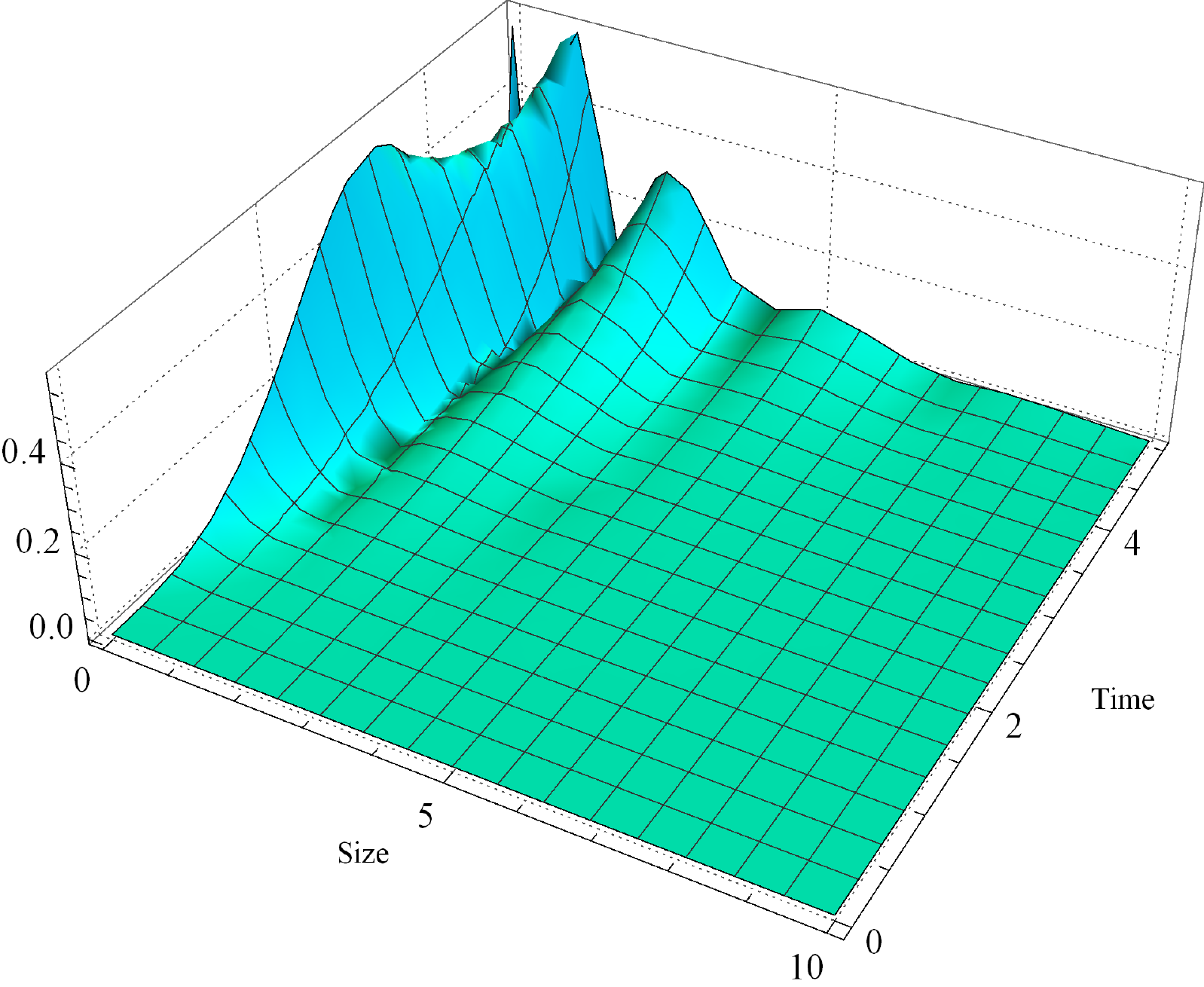}}
 	\subfigure[HPM error $(n=3)$]{\includegraphics[width=0.45\textwidth,height=0.35\textwidth]{{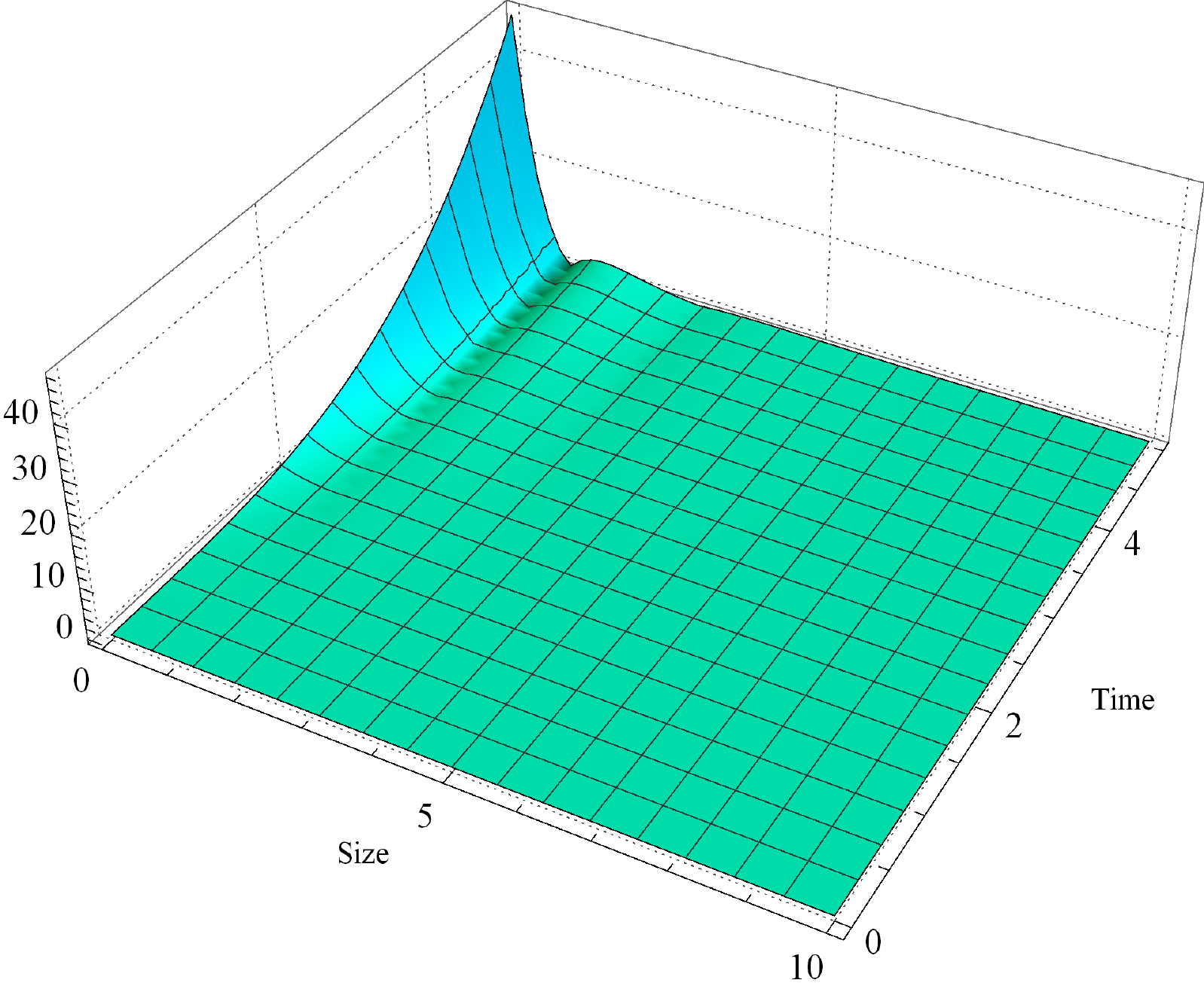}}}
 		\subfigure[ODM error $(n=3)$]{\includegraphics[width=0.45\textwidth,height=0.35\textwidth]{{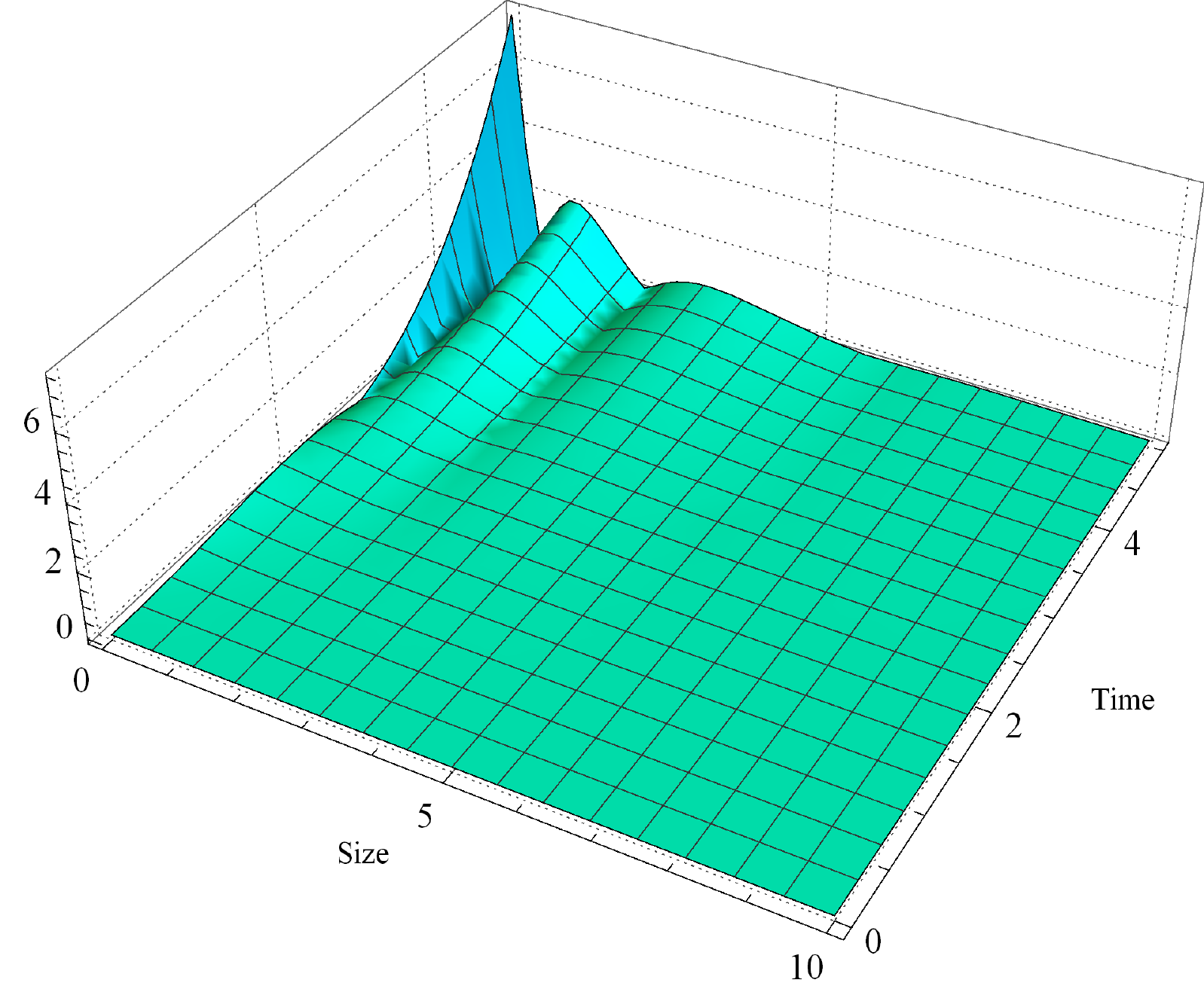}}}
\caption{AHPETM, HPM/ADM/HAM \& ODM errors}
\label{fig1}
\end{figure}
\begin{table*}[h]\centering
\caption{Numerical errors at $t=0.5,1,1.5 \text{ and } 2$ for $n=3,4,5,6$.}
\begin{tabular}{p{0.5cm}| p{2.5cm} p{2.5cm} p{2.5cm} p{2.5cm}}\toprule
$n$ & $t=0.5$ & $t=1$ & $t=1.5$ & $t=2$\\
\hline \hline
&&&&\\
3 & 0.0014 & 0.0153 & 0.0543 & 0.1239 \\
&&&&\\
4 & 1.366$\times 10^{-4}$ & 2.656$\times 10^{-3}$& 1.294$\times 10^{-2}$ & 3.632$\times 10^{-2}$\\
&&&&\\
5&$1.072\times 10^{-5}$ & 3.7972$\times 10^{-4}$ & 2.5718$\times 10^{-3}$& 9.0682$\times 10^{-3}$\\
&&&&\\
6&$7.154\times 10^{-7}$ & 4.6146$\times 10^{-5}$ & 4.3241$\times 10^{-4}$& 1.8931$\times 10^{-3}$\\
\bottomrule 
\end{tabular}
\label{error}
\end{table*}
Further, to see the beauty of our algorithm, errors between exact and AHPETM solutions are compared with the errors between exact and other well-established approximated solutions obtained via HPM/ADM/HAM and ODM in Figure \ref{fig1}. It is important to point out here that the HAM  \cite{kaur2022approximate}/ADM \cite{singh2015adomian}/HPM \cite{kaur2019analytical} provide the same iterations and hence the identical finite term series solution for the considered model. From Figure \ref{fig1}, it is observed that HAM is very badly approximated in comparison to ODM and AHPETM further improves the error of ODM very significantly. Table \ref{error} depicts the numerical errors of AHPETM at different time levels for $n=3,4,5$ and 6 using the formula provided in \cite{singh2015adomian}. As one can notice, the inaccuracy grows as time increases for a fixed number of terms and error decreases when more terms in the approximated solutions are taken into account.\\
 \begin{figure}[htb!]
 		\centering
	\subfigure[Number density at $t=2$]{\includegraphics[width=0.45\textwidth,height=0.35\textwidth]{{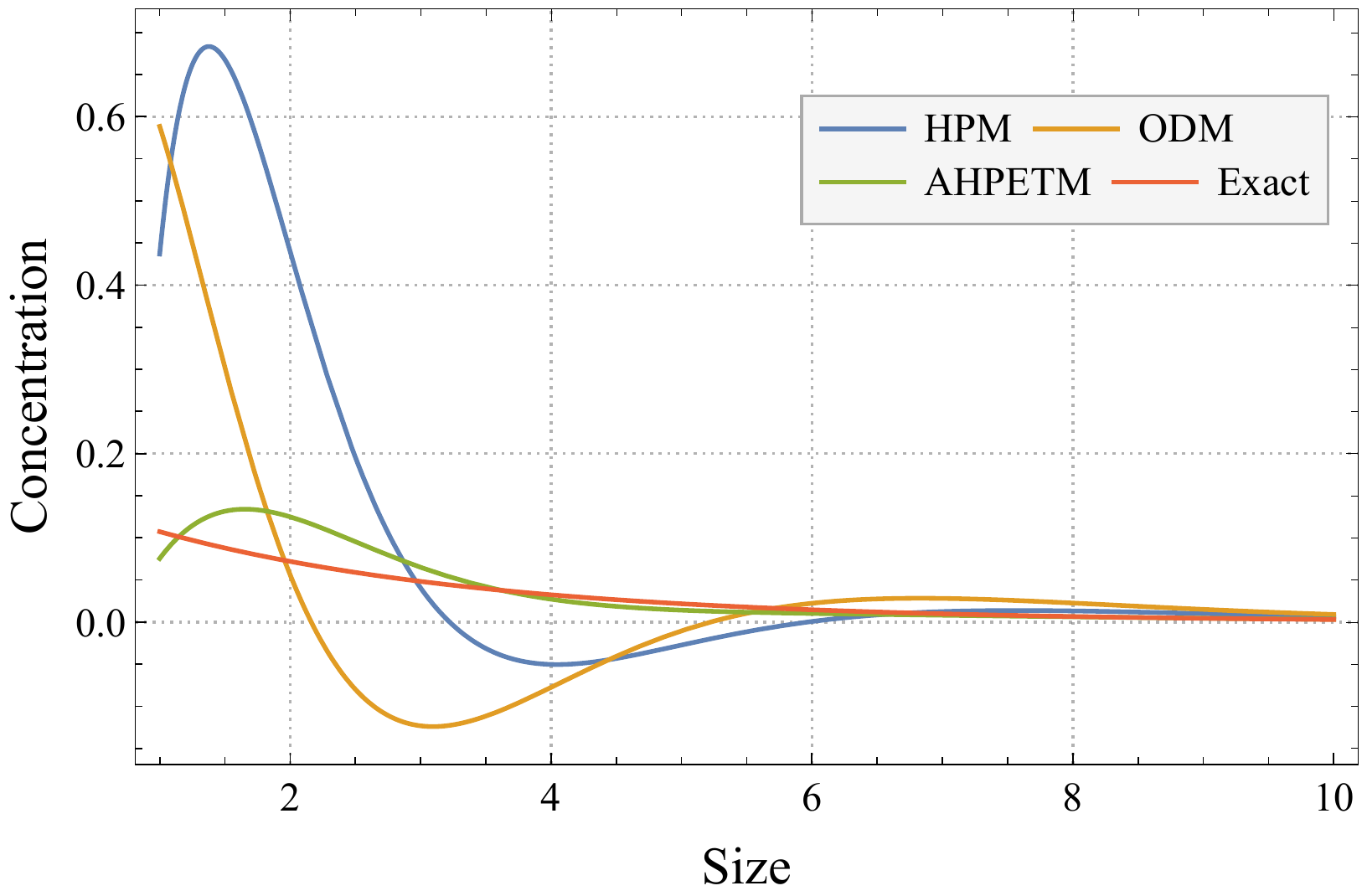}}}
 		\subfigure[Error at $x=5$]{\includegraphics[width=0.45\textwidth,height=0.35\textwidth]{{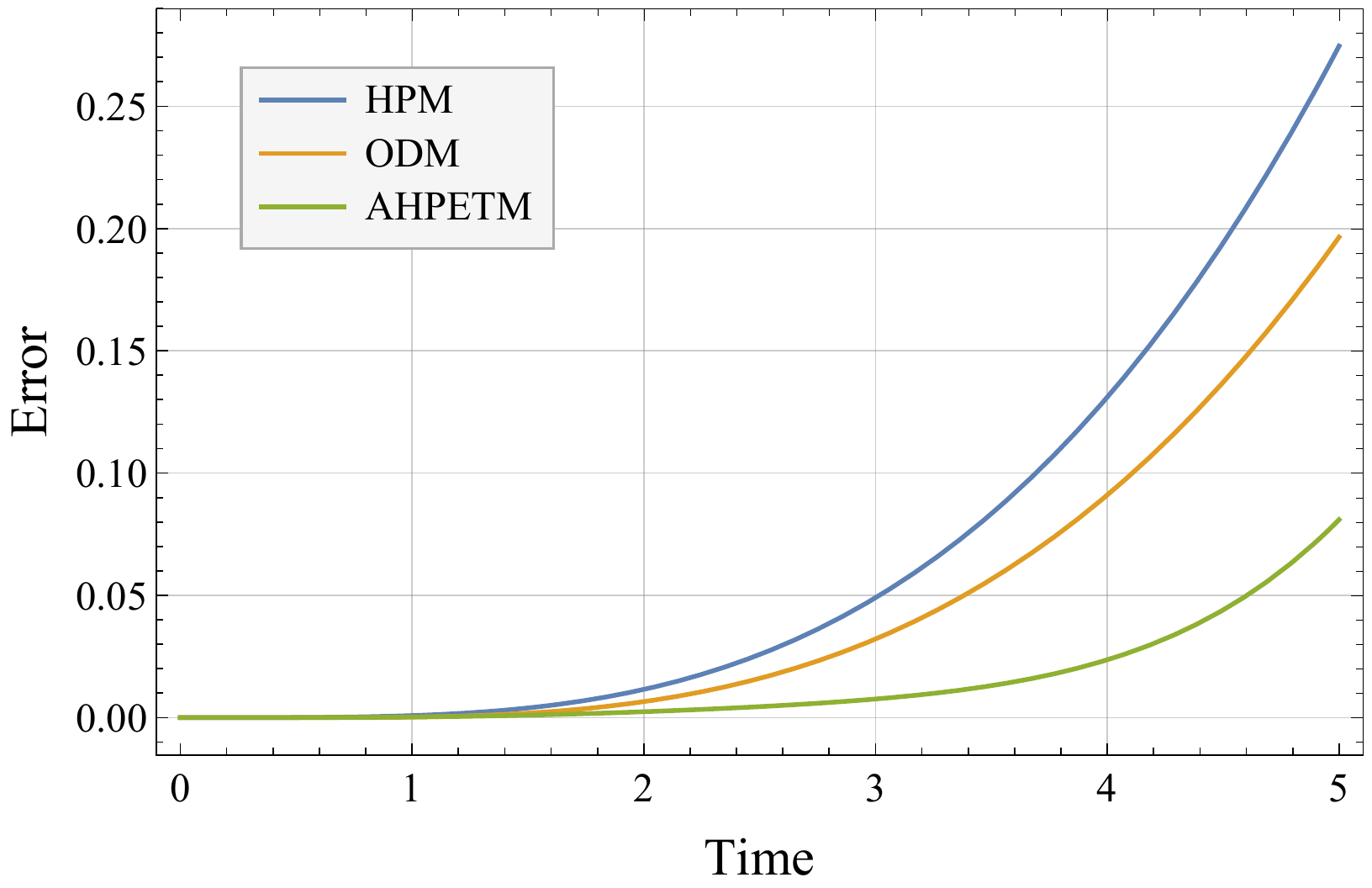}}}
\caption{Number density and error}
\label{fig2}
\end{figure}
 The superiority of AHPETM over HAM and ODM is also demonstrated in Figure \ref{fig2}. Figure \ref{fig2}(a) represents the concentration of particles at time $t=2$ and the solutions obtained using HAM and ODM blow up where the AHPETM solution matches well with the exact solution. Results for error from  Figure \ref{fig2}(b) indicate that all the schemes are quite efficient for a short time where as for a significant time, the errors due to HAM and ODM are relatively very high compared to AHPETM.
\begin{figure}[htb!]
 		\centering
	\subfigure[Zeroth moments]{\includegraphics[width=0.45\textwidth,height=0.35\textwidth]{{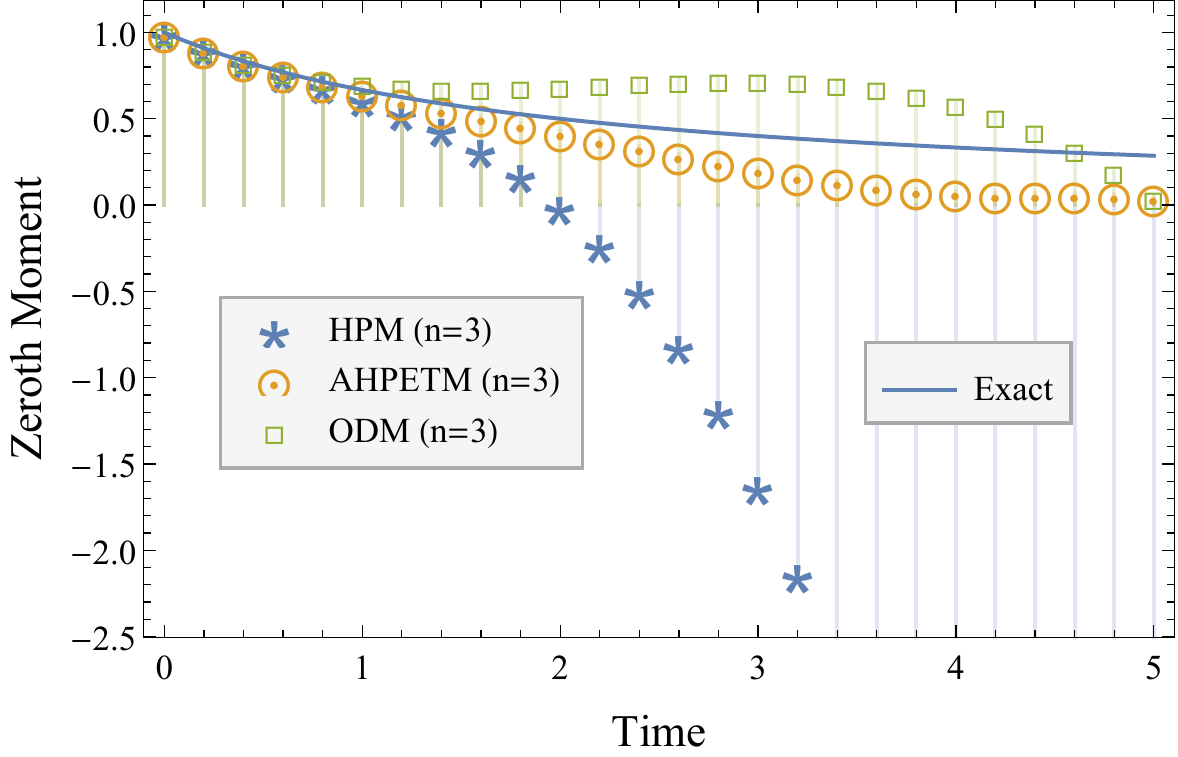}}}
 		\subfigure[Second moments]{\includegraphics[width=0.45\textwidth,height=0.35\textwidth]{{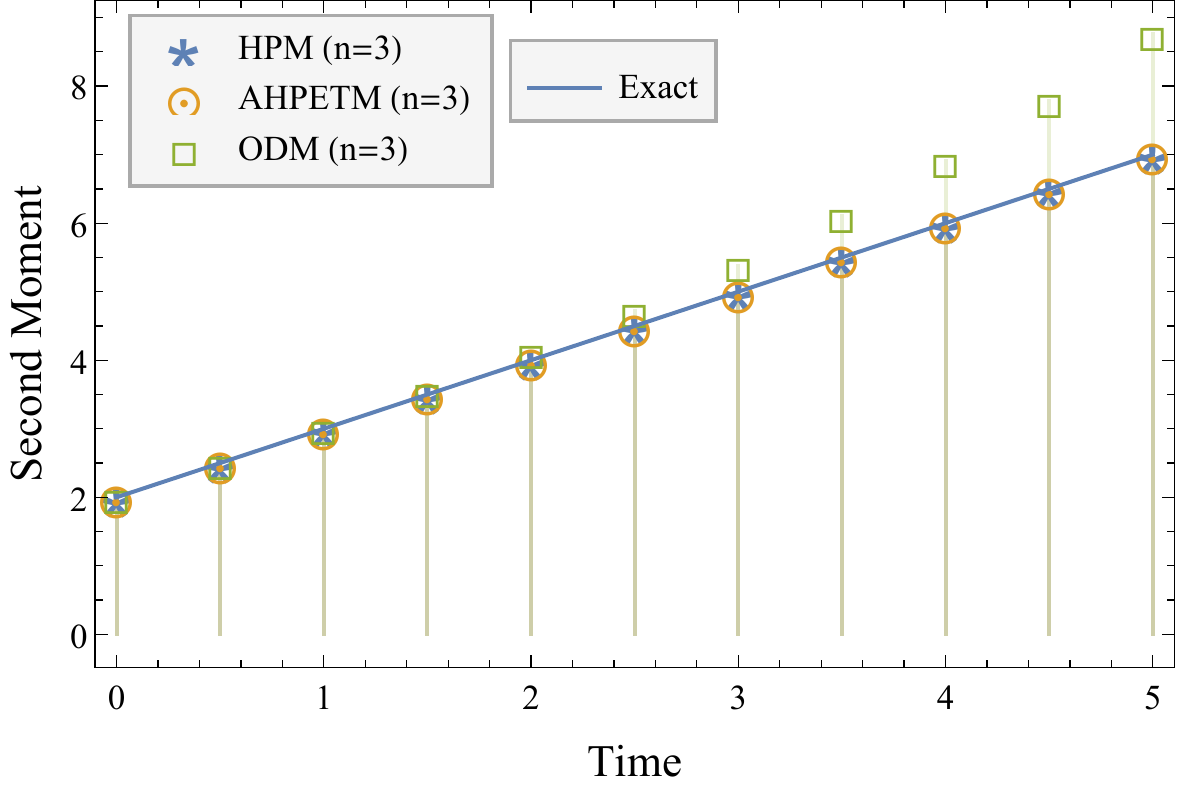}}}

\caption{Zeroth and second moments}
\label{fig3}
\end{figure}
Further, integral properties associated with number density are plotted in Figure \ref{fig3}. The zeroth (total number of particles) and second (energy dispersed by the system) moments are displayed and comparison are made with the precise moments. In Figure \ref{fig3}(a), AHPETM offers superior approximations in the zeroth moment while HAM under predicts the result and deviates almost exponentially from the exact one. ODM shows much better approximation than HAM but still suffers fluctuations. As shown in Figure \ref{fig3}(b), AHPETM continues to be the best option as the second moment produced by AHPETM and HAM coincides with the exact ones but ODM does not offer a decent estimate.\\
\begin{example}
Let us take aggregation kernel $K(x,y)=x+y$ with the exponential initial condition $u(x,0)=e^{-x}$. The exact number density is provided in \cite{scott1968analytic} as
$$u(x,t)=\frac{e^{\left(e^{-t}-2\right) x-t} I_1\left(2 \sqrt{1-e^{-t}} x\right)}{\sqrt{1-e^{-t}} x},$$ where $I_1$ is the Bessel function of the first kind.
\end{example}
Using the equations \eqref{agg_iterations} and \eqref{agg_trn}, first few components of the series solution are determined as
$$v_0(x,t)=e^{-x}, \quad v_1(x,t)= \frac{1}{2} t e^{-x} \left( x^2-2 x-2\right) ,$$
\begin{align*}
v_2(x,t)= \frac{1}{720} t^2 e^{-x} \left(t x (x^5-10  x^4-20  x^3+240  x^2-120 x-240)+60 x^4-360 x^3-180 x^2+1080 x+360\right).
\end{align*}
Continuing in a similar fashion, it is easy to compute the higher order components to find better-approximated results. A four-term truncated solution is considered here and the results are compared with the HAM and ODM solutions using the same number of terms.
		\begin{figure}[htb!]
	 		\centering
	 	\subfigure[AHPETM error $(n=4)$]{\includegraphics[width=0.45\textwidth,height=0.35\textwidth]{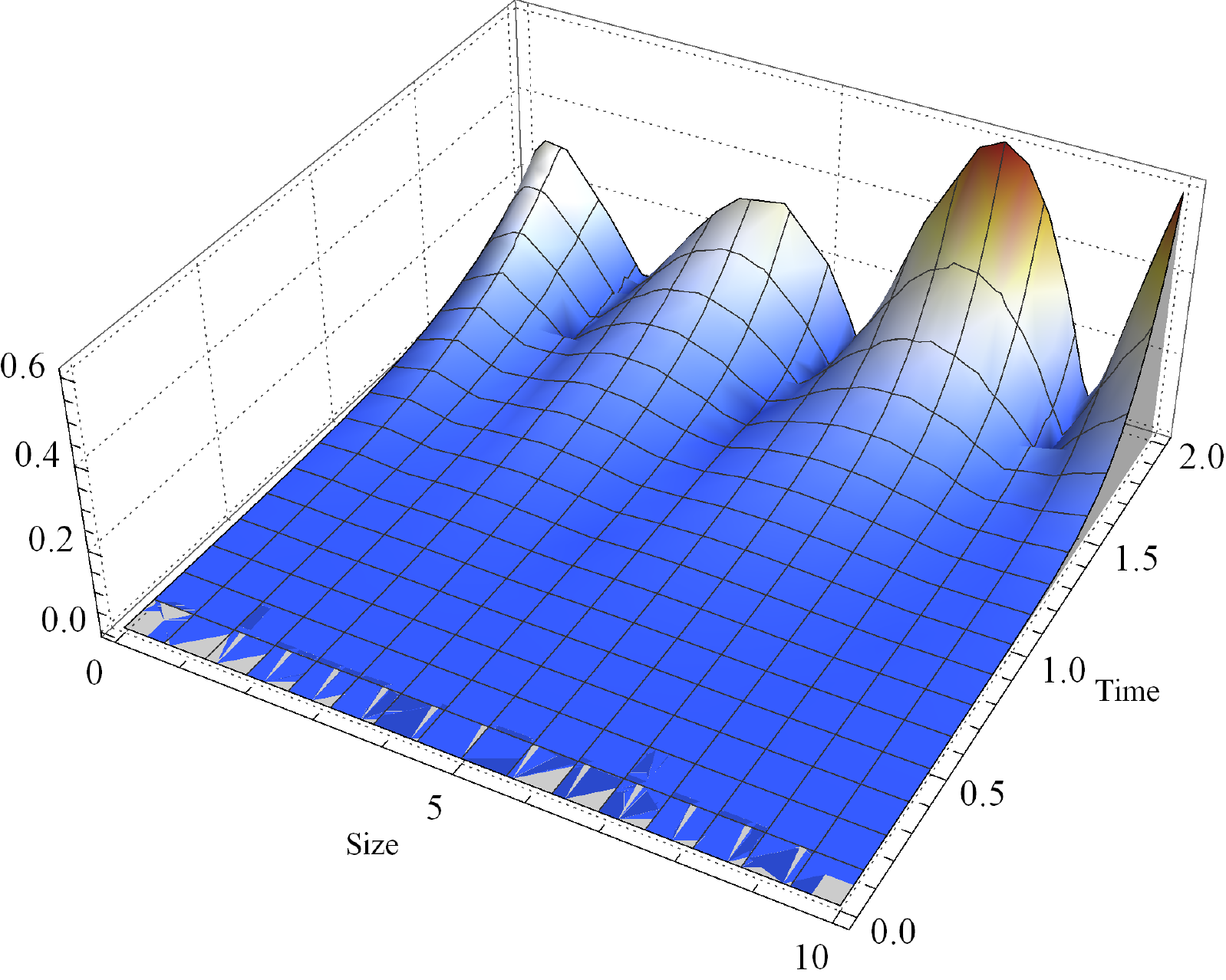}}
	 	\subfigure[HPM error $(n=4)$]{\includegraphics[width=0.45\textwidth,height=0.35\textwidth]{{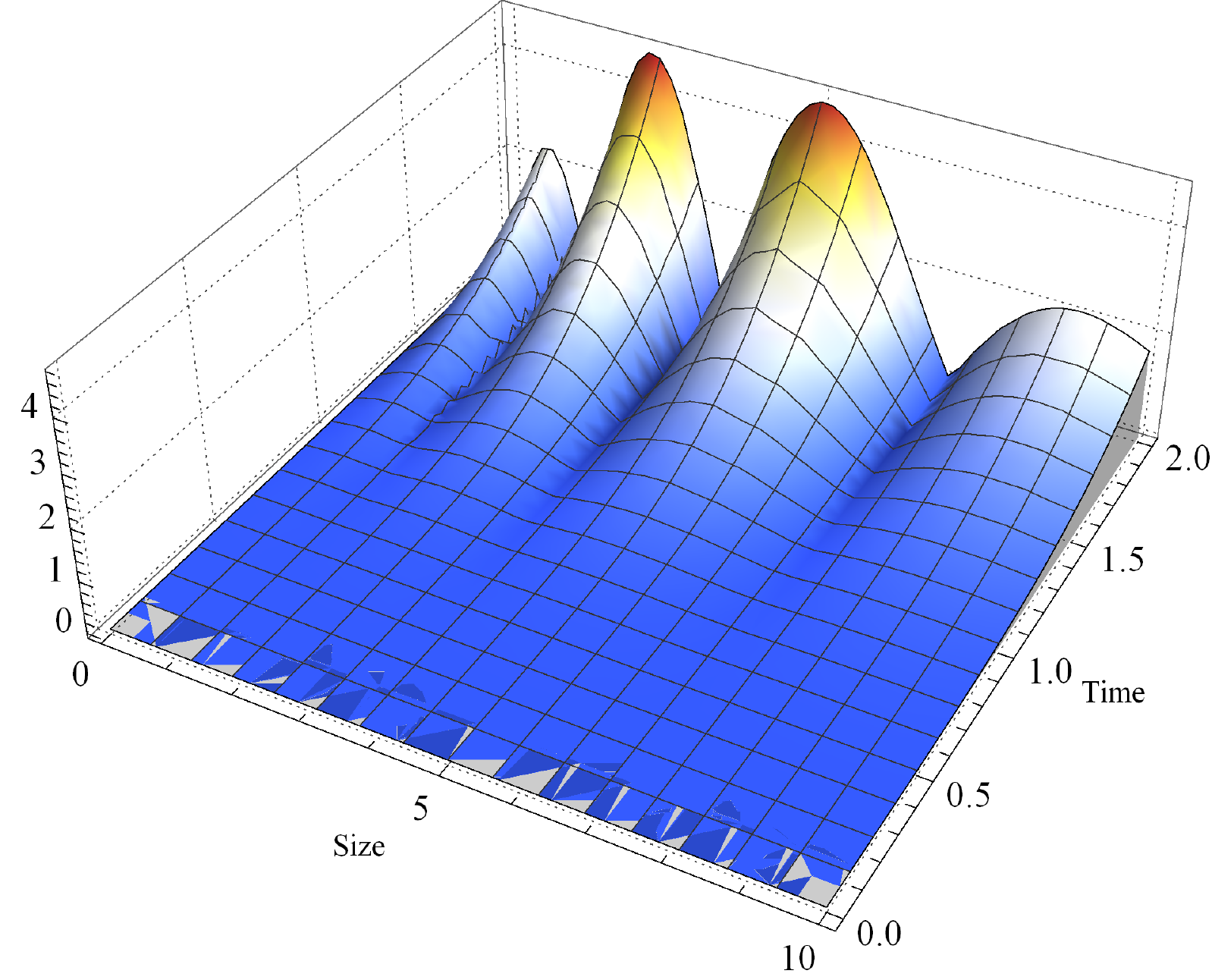}}}
	 		\subfigure[ODM error $(n=4)$]{\includegraphics[width=0.45\textwidth,height=0.35\textwidth]{{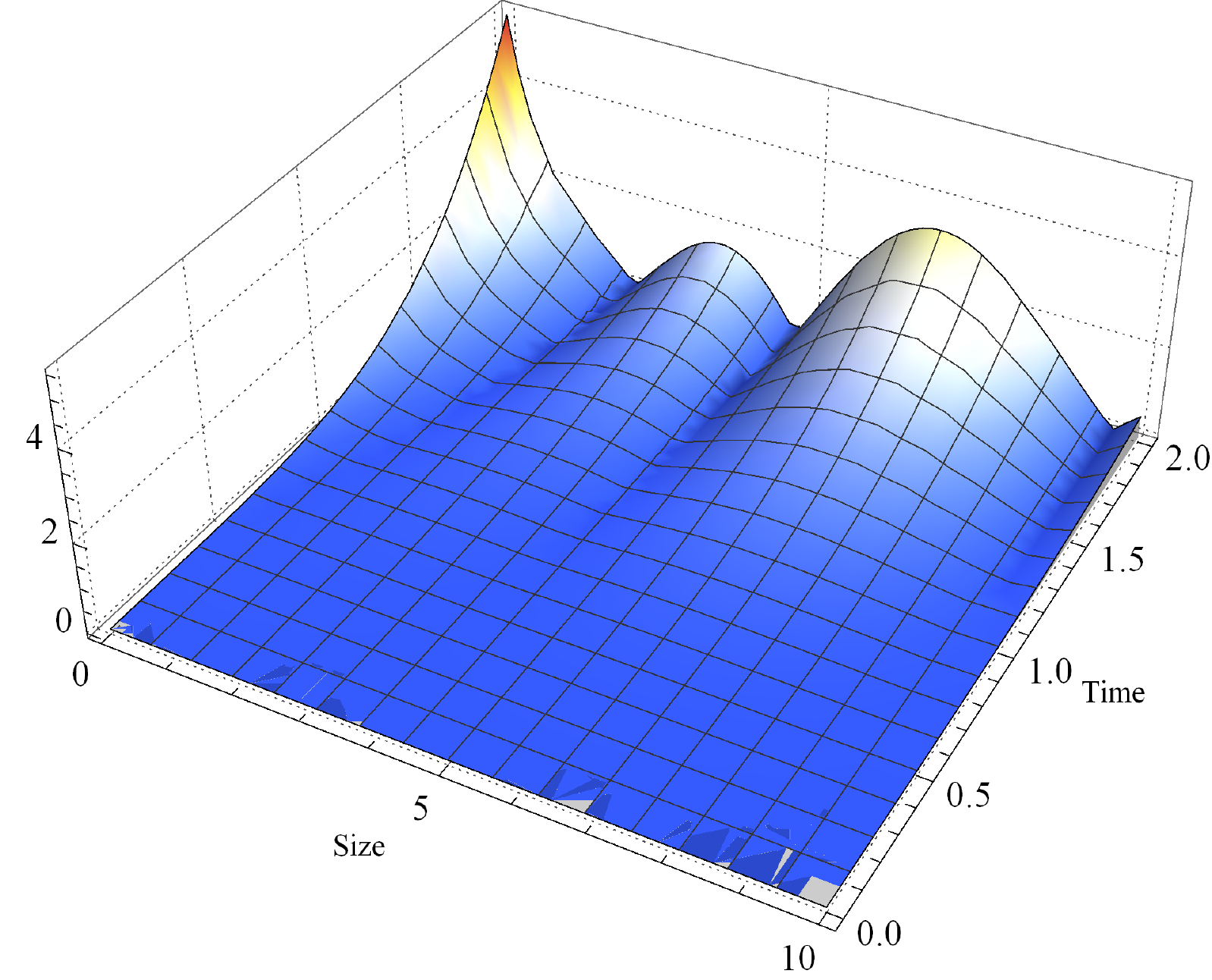}}}
	\caption{AHPETM, HPM/ADM/HAM \& ODM errors}
	\label{fig4}
	\end{figure}
	\begin{figure}[htb!]
	 		\centering
		\subfigure[Number density at $t=2$]{\includegraphics[width=0.45\textwidth,height=0.35\textwidth]{{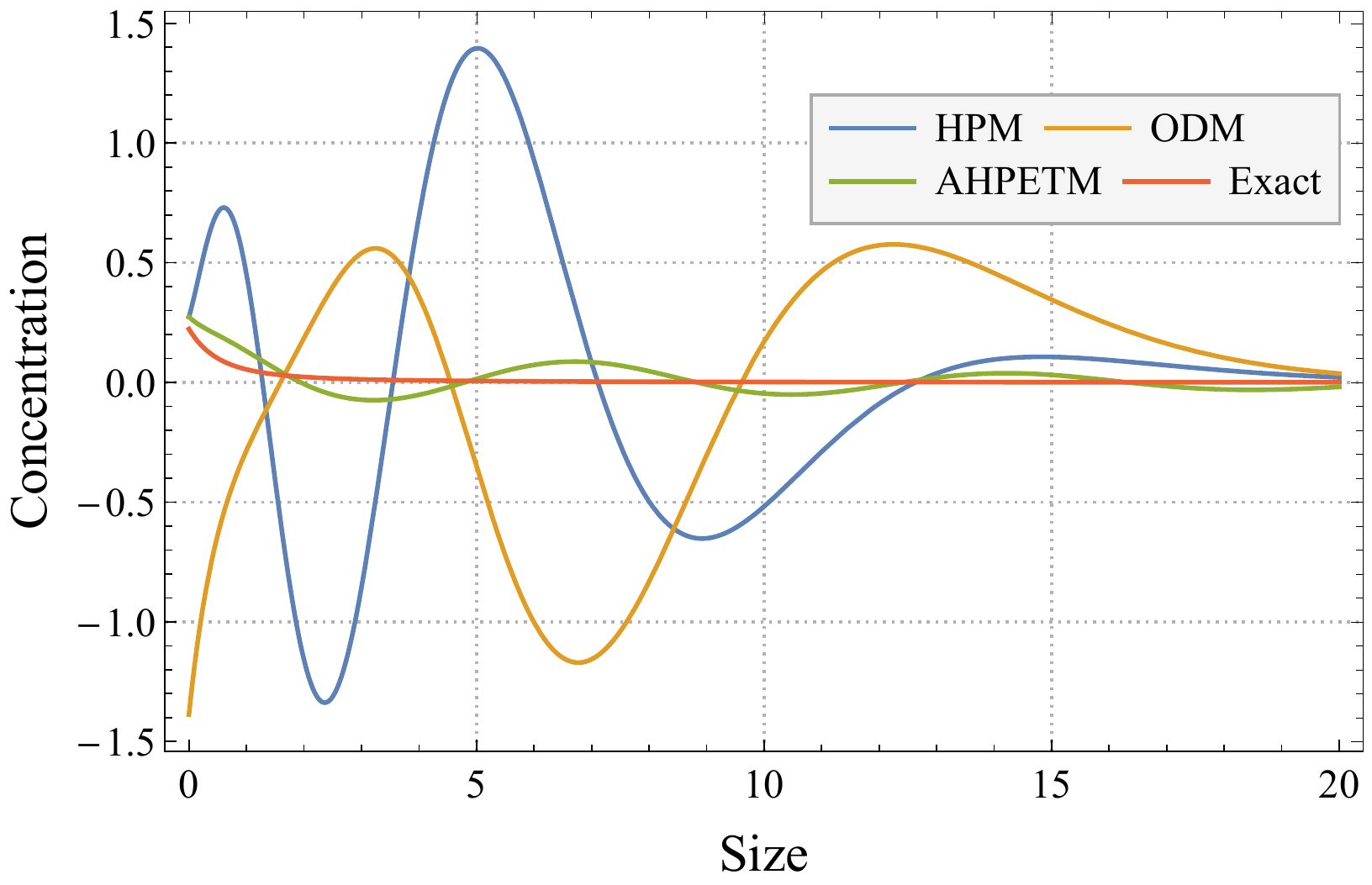}}}
	 		\subfigure[Error at $x=5$]{\includegraphics[width=0.45\textwidth,height=0.35\textwidth]{{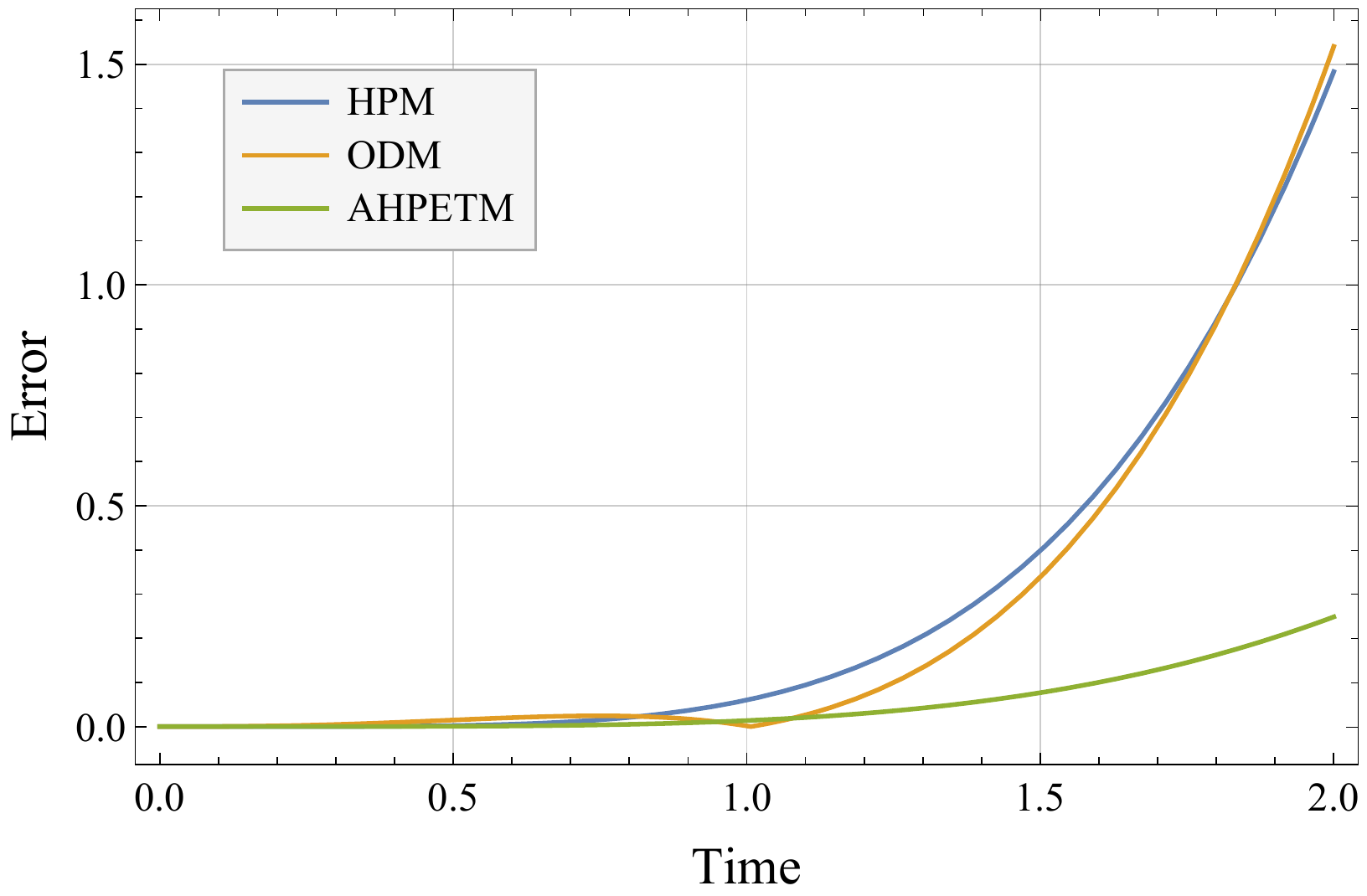}}}
	
	\caption{Number density and error}
	\label{fig5}
	\end{figure}\\
	\begin{table*}[h]\centering
	\caption{Absolute error for $x=5$ at different time levels}
	\begin{tabular}{p{0.5cm}| p{1.5cm} p{1.8cm} p{1.5cm} p{1.5cm} p{3cm} p{3cm} p{3cm}}\toprule

	$t$ & Exact & AHPETM & ODM & HPM & AHPETM error& ODM error & HPM error\\
			\hline \hline
	&&&&&&&\\
	0.2& 0.0129& 0.0129& 0.0118 & 0.0131& 2.71288$\times 10^{-5}$&1.0352$\times 10^{-3}$&1.4248$\times 10^{-4}$\\
	0.4& 0.0146& 0.0151& 0.0053 & 0.0184& 5.3035$\times 10^{-4}$& 9.2238$\times 10^{-3}$& 3.8424$\times 10^{-3}$\\
	0.6& 0.0138& 0.0162& -0.0177 & 0.03868& 2.3932$\times 10^{-3}$& 3.1524$\times 10^{-2}$& 2.4873$\times 10^{-2}$\\
	0.8& 0.0121& 0.0179& -0.0601 & 0.1026& 5.8862$\times 10^{-3}$& 7.223$\times 10^{-2}$& 9.060 $\times 10^{-2}$\\
	1.0& 0.0101& 0.0203& -0.1226 & 0.2523& 0.0102& 0.1327& 0.2422\\
	1.2& 0.0082& 0.0220& -0.2034 & 0.5428& 0.0137& 0.2117& 0.5345\\
	1.4& 0.0067& 0.0197&-0.2989 & 1.042& 0.0131 & 0.3057& 1.0357\\
	1.6& 0.00545& 0.00506& -0.4031 & 1.8325& 0.00038& 0.4085 & 1.8271\\
		\bottomrule 
	\end{tabular}
	\label{error1}
	\end{table*}
As observed in the previous case, AHPETM is again found to be more accurate than HAM and ODM, see Figure \ref{fig4}.  The error due to AHPETM is not only significantly smaller than the existing approximated solutions of HAM and ODM but also close to zero. This is also clear from Figure \ref{fig5}(a) which displays the error at $t=2$ for all the schemes. Further, Figure \ref{fig5}(b) and Table \ref{error1} demonstrate that for a fixed $x$ and time upto $t=1$, HPM and ODM errors are almost identical and insignificant but as the time increases, the error due to the first two schemes are still the same but grow almost exponentially while AHPETM performs almost consistently. 
	\begin{figure}[htb!]
	 		\centering
   		\subfigure[Zeroth moments]{\includegraphics[width=0.45\textwidth,height=0.35\textwidth]{{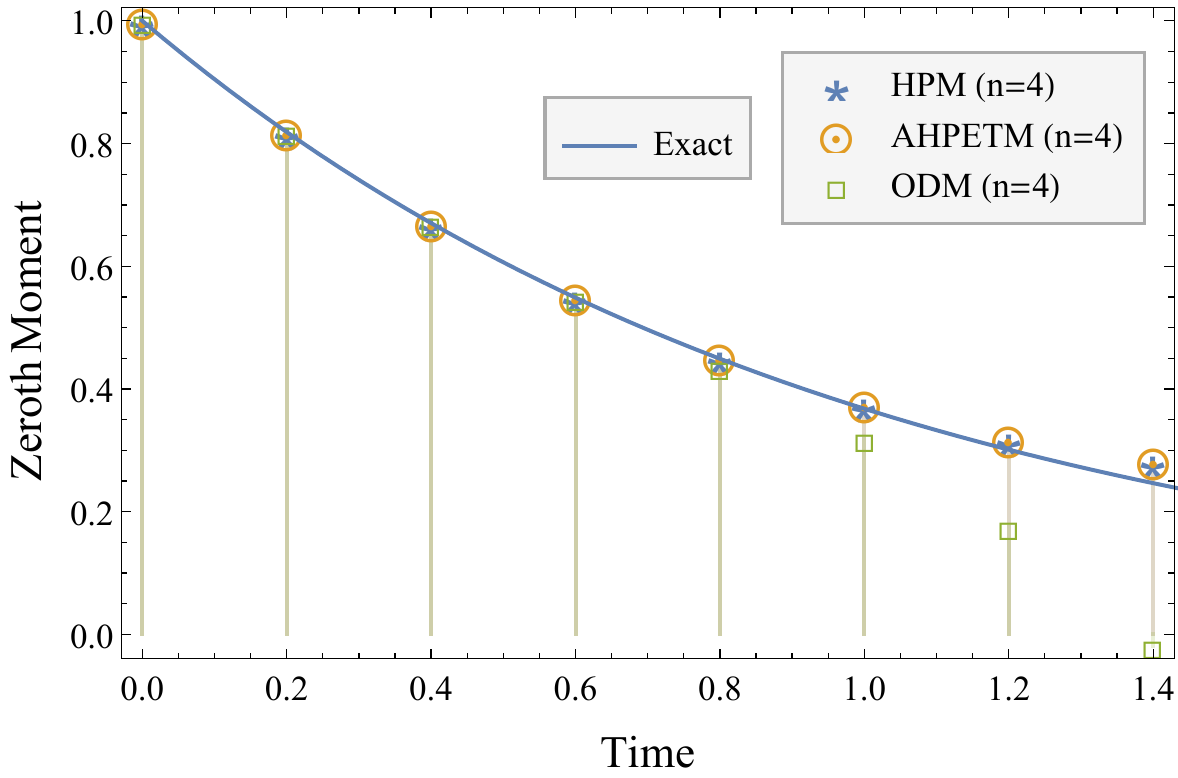}}}
	 		\subfigure[Second moments]{\includegraphics[width=0.45\textwidth,height=0.35\textwidth]{{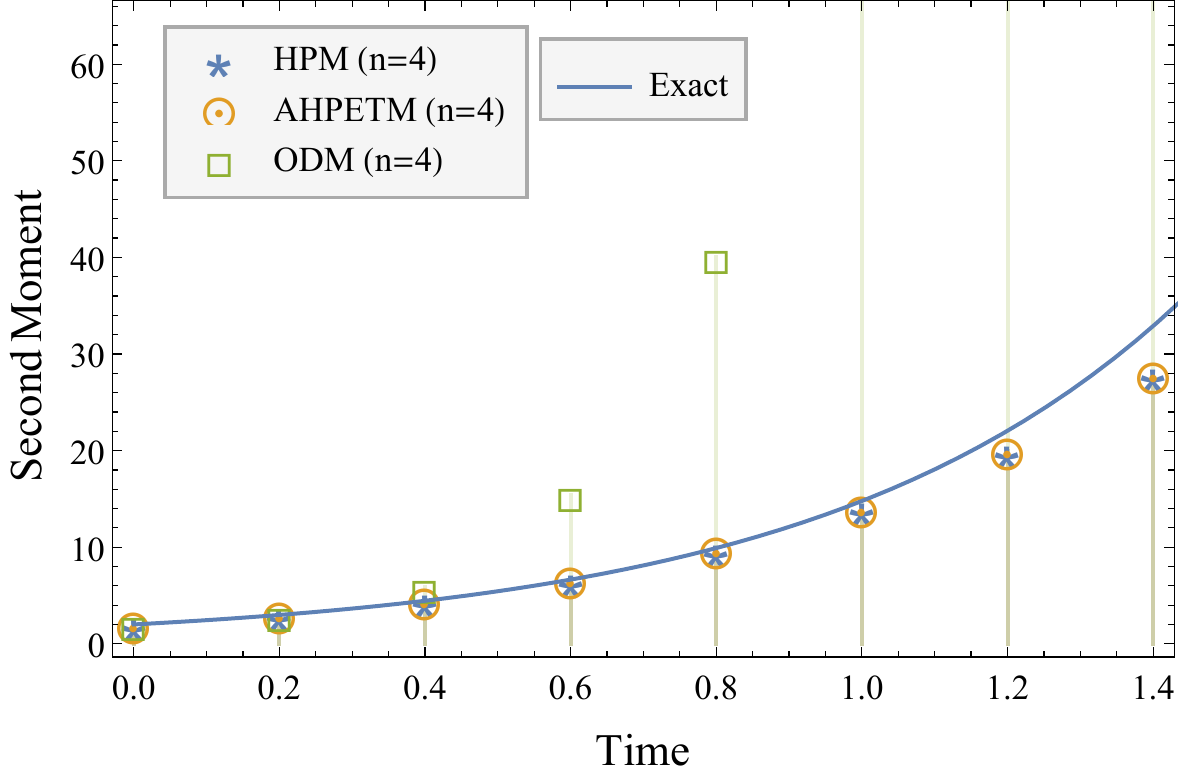}}}
	\caption{Zeroth and second moments}
	\label{fig6}
	\end{figure}\\
Moving further, approximated and exact moments are compared for all three methods in Figure \ref{fig6}. Surprisingly, ODM under predicts and over predicts the zeroth and second moment, respectively, while HPM and AHPETM gave almost identical findings and provided an excellent approximations to the exact zeroth and second moments. We would like to point out that the first moment is constant in all cases therefore, graphs are omitted. 
\begin{example}
Consider the case of product aggregation kernel $K(x,y)=xy$ with the exponential initial condition $u(x,0)=e^{-x}$ and the precise solution is provided in \cite{ranjbar2010numerical} as
$$u(x,t)=\sum _{k=0}^{\infty} \frac{t^k x^{3 k} \exp (-(t+1) x)}{(k+1)! \Gamma (2 k+2)}.$$ 
\end{example}
Using the recursive scheme defined in equation \eqref{agg_iterations}, a five term truncated solution is considered. Due to complexity of the terms, only few given here as
$$v_0(x,t)=e^{-x}, \quad v_1(x,t)= \frac{1}{12} t e^{-x} x \left(x^2-12\right),$$
\begin{align*}
v_2(x,t)=\frac{1}{544320}\left(t^2 e^{-x} x^2 \left(t x^7-144 t x^5+3024 t x^3+756 x^4-45360 x^2+272160\right)\right).
\end{align*}
		\begin{figure}[h]
	 		\centering
	 	\subfigure[AHPETM error $(n=5)$]{\includegraphics[width=0.45\textwidth,height=0.35\textwidth]{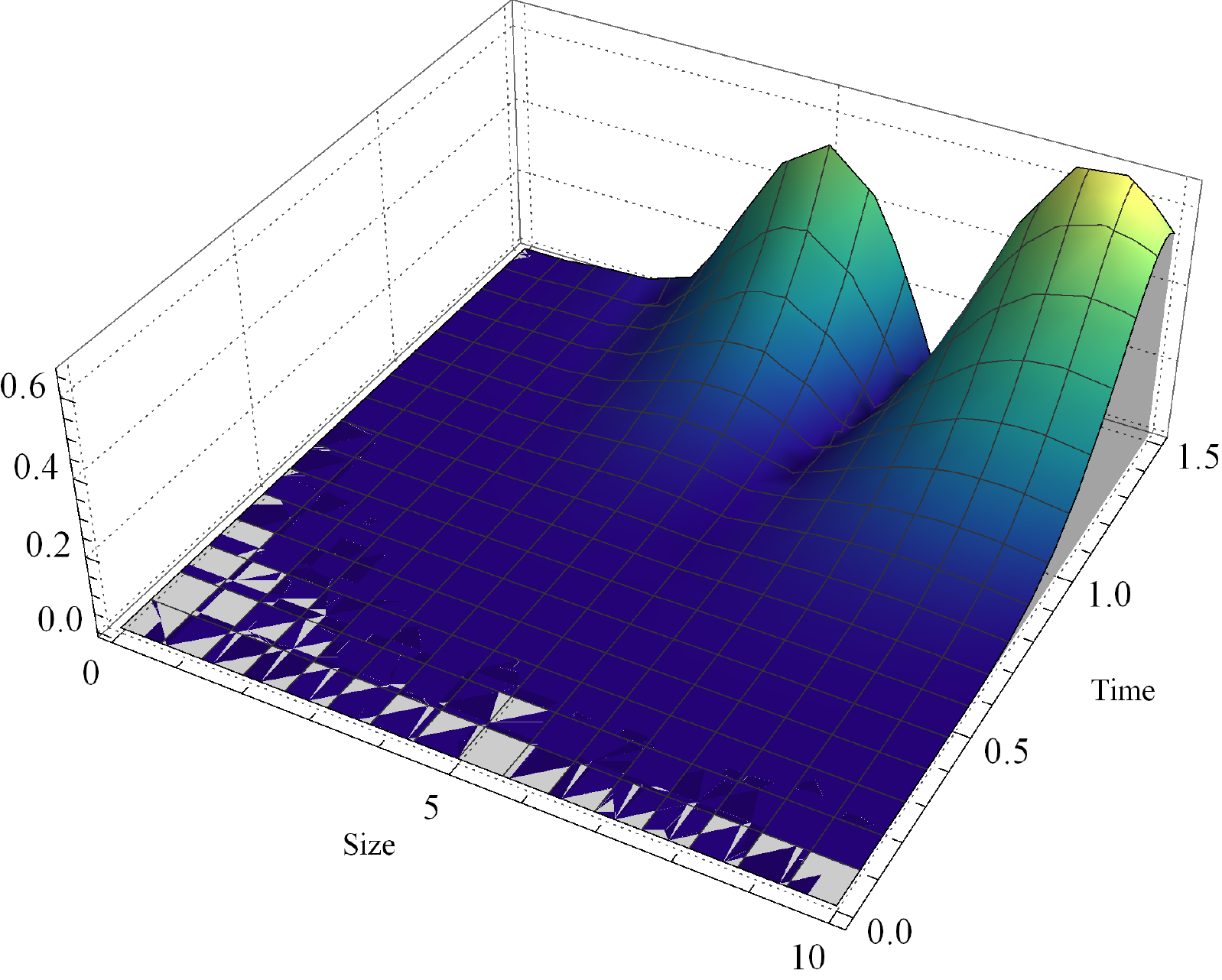}}
	 	\subfigure[HPM error $(n=5)$]{\includegraphics[width=0.45\textwidth,height=0.35\textwidth]{{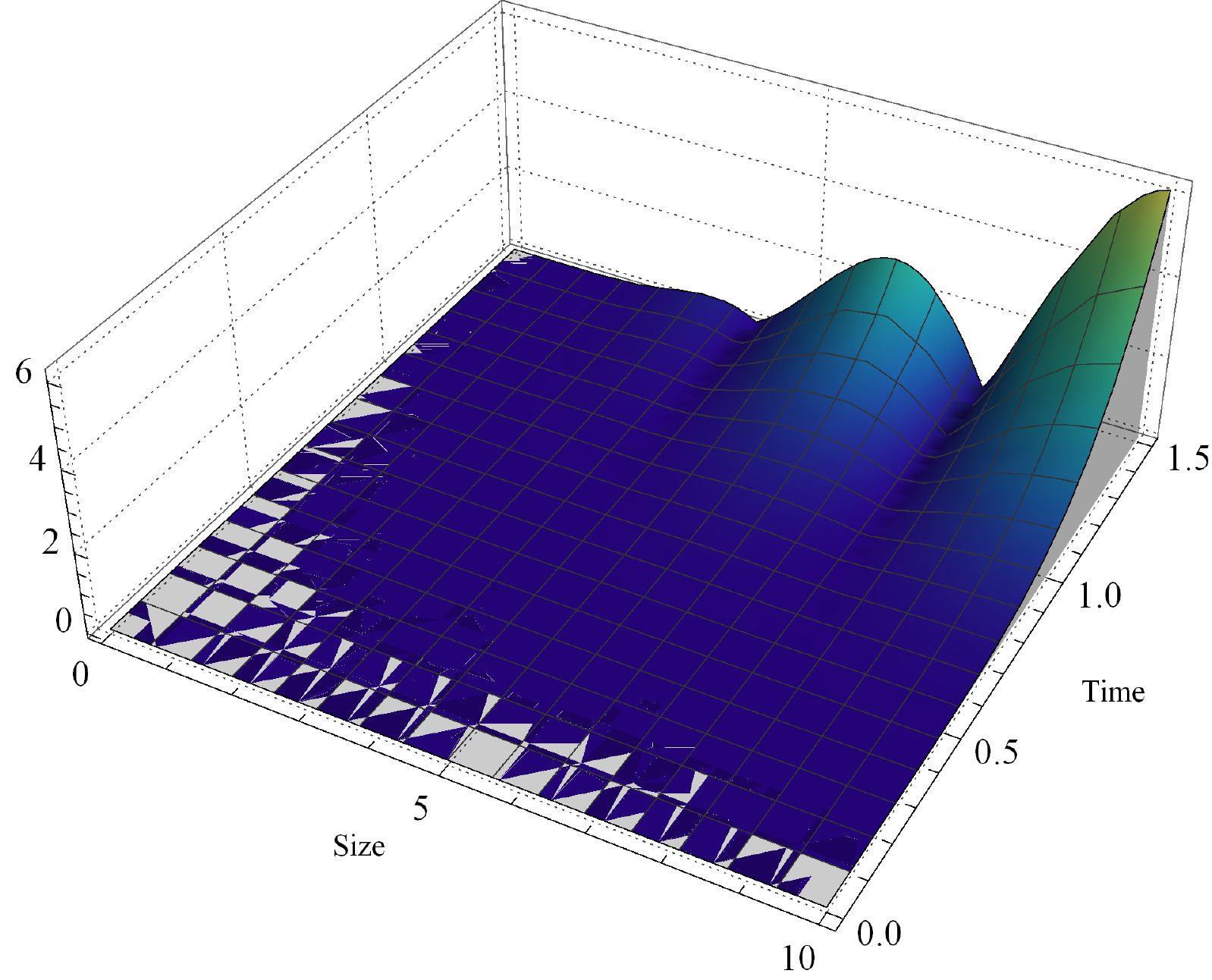}}}
	 		\subfigure[ODM error $(n=5)$]{\includegraphics[width=0.45\textwidth,height=0.35\textwidth]{{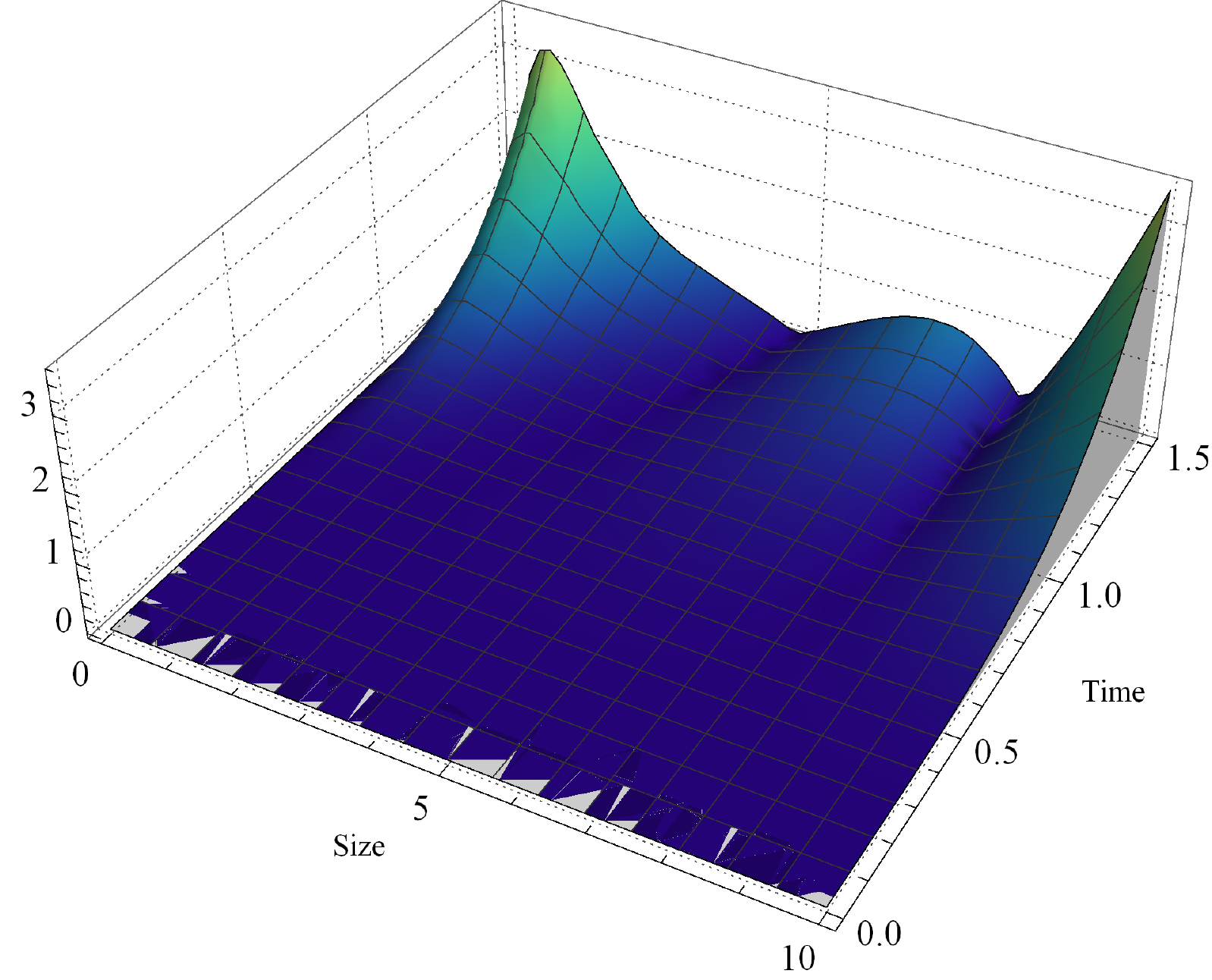}}}
	\caption{AHPETM, HPM/ADM/HAM \& ODM errors}
	\label{fig7}
	\end{figure}
	\begin{figure}[htb!]
	 		\centering
		\subfigure[Number density at $t=2$]{\includegraphics[width=0.45\textwidth,height=0.35\textwidth]{{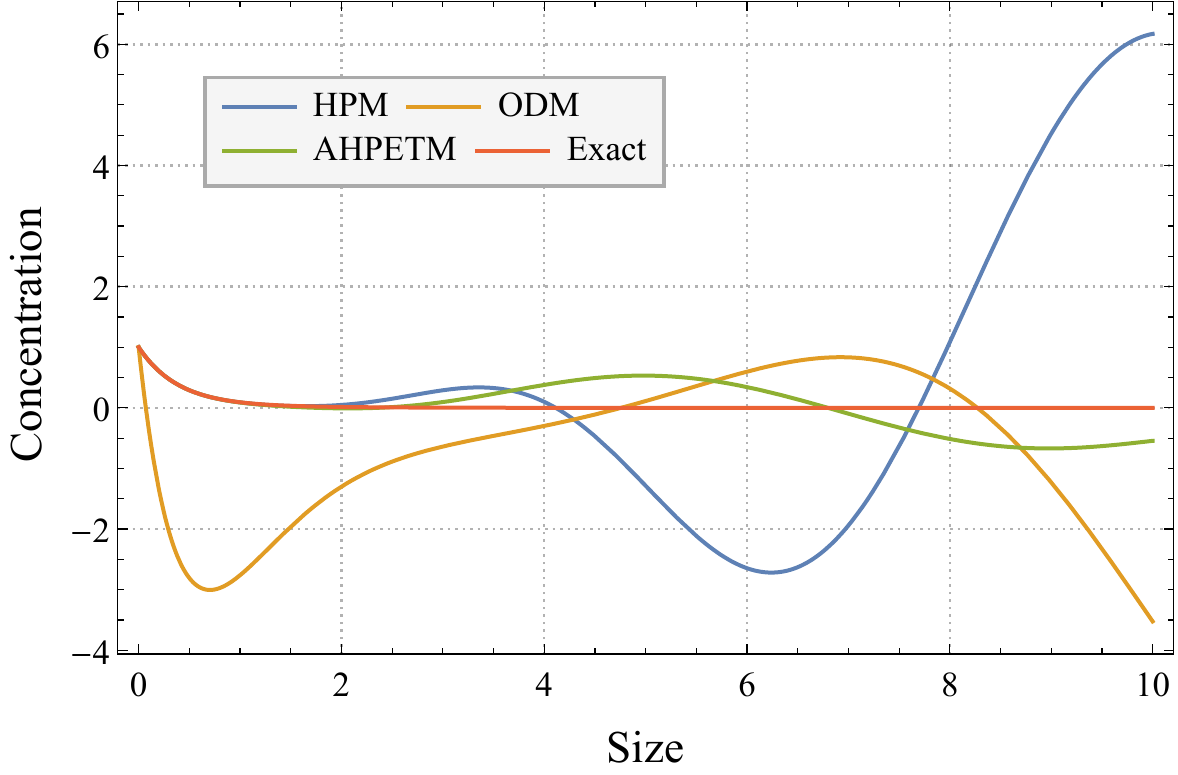}}}
	 		\subfigure[Error at $x=5$]{\includegraphics[width=0.45\textwidth,height=0.35\textwidth]{{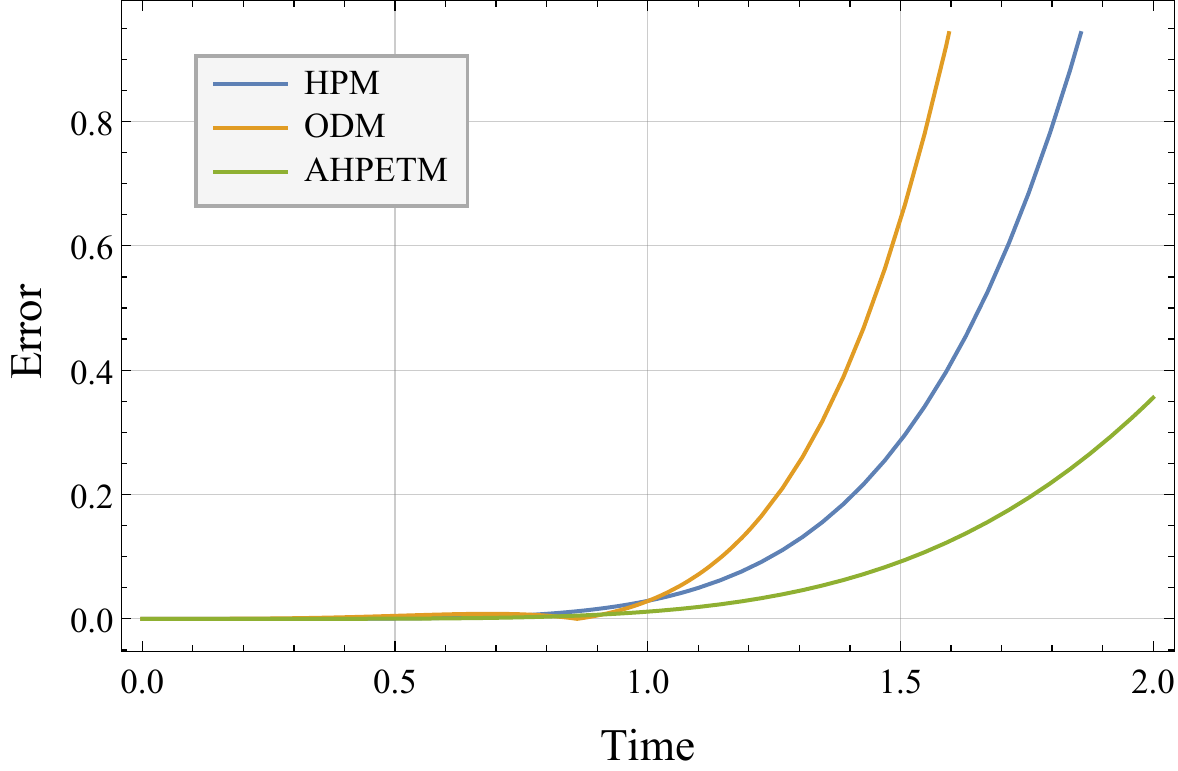}}}
	
	\caption{Number density and error}
	\label{fig8}
	\end{figure}
	\begin{figure}[htb!]
	 		\centering
		\subfigure[Zeroth moments]{\includegraphics[width=0.45\textwidth,height=0.35\textwidth]{{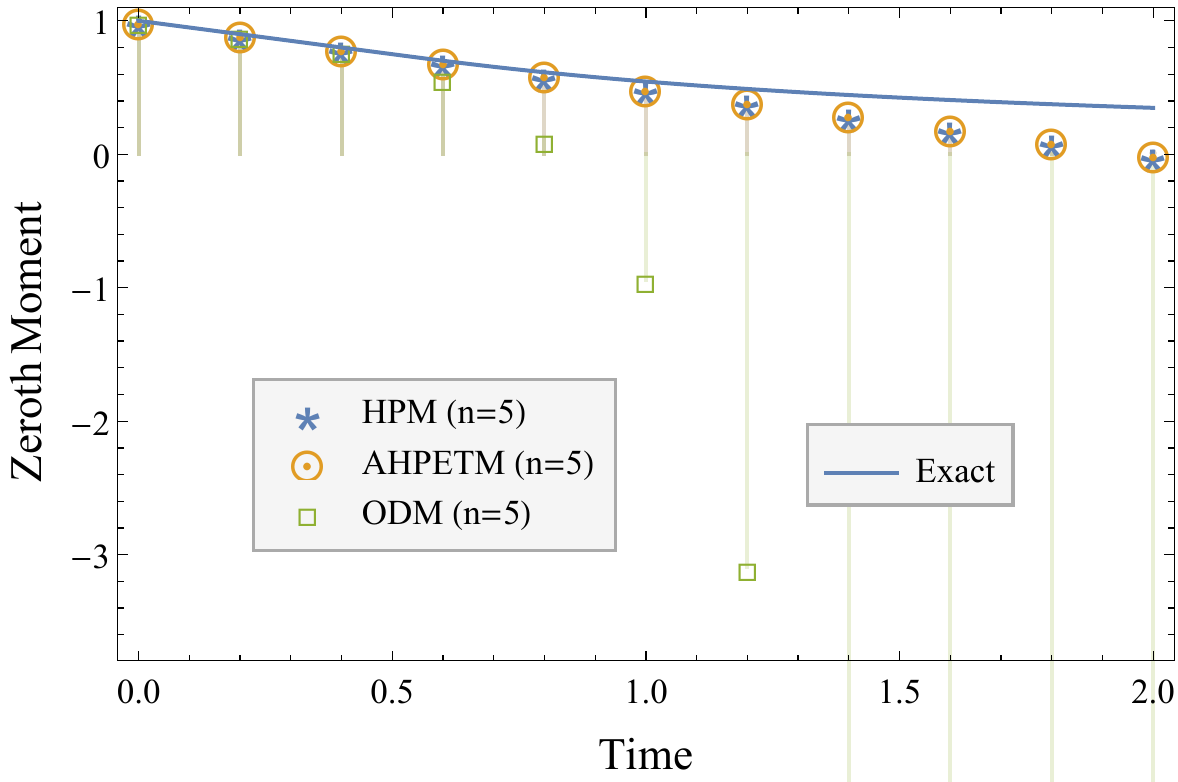}}}
	 		\subfigure[Second moments]{\includegraphics[width=0.45\textwidth,height=0.35\textwidth]{{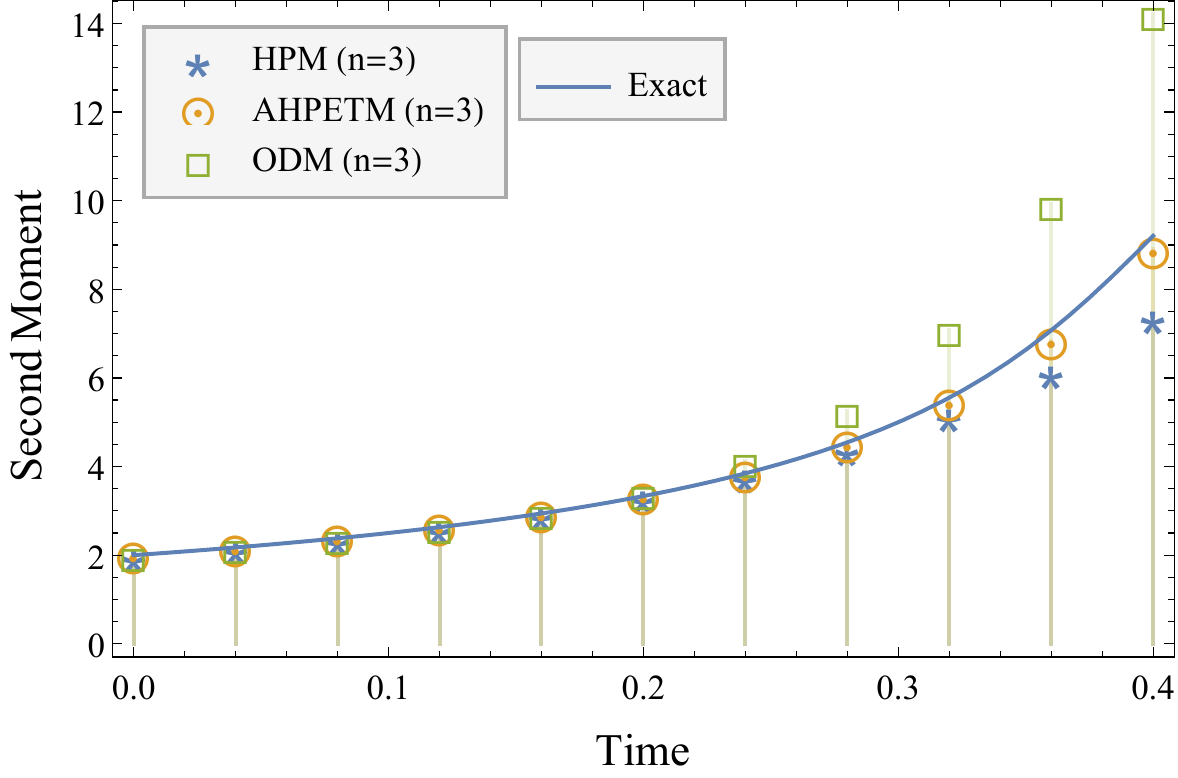}}}
	
	\caption{Zeroth and second moments}
	\label{fig9}
	\end{figure}\\
Figure \ref{fig7} contrasts the error between the exact and truncated solutions obtained via HAM/HPM/ADM, ODM, and AHPETM. The figure demonstrates that the AHPETM outperforms both HPM and ODM results. Figure \ref{fig8} further indicates the scheme's superiority as the unexpected behavior of the HPM and ODM solutions are noticed, where as the AHPETM offers better estimates of the exact solution. Figure \ref{fig9} continues by contrasting the analytical moments with the approximated moments. In situations where the HPM and AHPETM moments are almost identical and closer to exact moments, as seen in the prior occurrences, ODM moments explode.
\par From the above illustrations, it can be seen that in all contexts, AHPETM performs better than ADM, HPM, HAM, and ODM. Therefore, due to the novelty of the proposed scheme, we use AHPETM to solve the more complex models such as coupled aggregation breakage and bivariate aggregation equations.
\subsection{Coupled Aggregation-Breakage Equation}
\begin{example}
Considering the case of constant aggregation rate $(K(x,y)=1)$, binary breakage $(B(x,y)=2/y)$ with the selection rate $S(x)=\frac{x}{2}$ and for the initial condition $u(x,0)= 4 x e^{-2 x}$, the exact solution  for the problem \eqref{agg_break_eqn} is provided in \cite{lage2002comments}.  
\end{example}
Using the iterations defined in equation \eqref{agg_break_iterations}, components $v_i's$ of the solutions are computed as follows
$$v_0(x,t)= 4 x e^{-2 x}, \quad v_1(x,t)=\frac{1}{3} t e^{-2 x} \left(4 x^3-6 x^2-6 x+3\right),$$
\begin{align*}
v_2(x,t)=&\frac{1}{3780}\bigg(t^2 e^{-2 x} \bigg(8 t x^7-56 t x^6-84 t x^5+840 t x^4-420 t x^3-1260 t x^2+630 t x+504 x^5\\&-2520 x^4-1890 x^3+9450 x^2+945 x-1890\bigg)\bigg).
\end{align*}
\begin{figure}[htb!]
 		\centering
 	\subfigure[Number density  $(n=4)$]{\includegraphics[width=0.45\textwidth,height=0.35\textwidth]{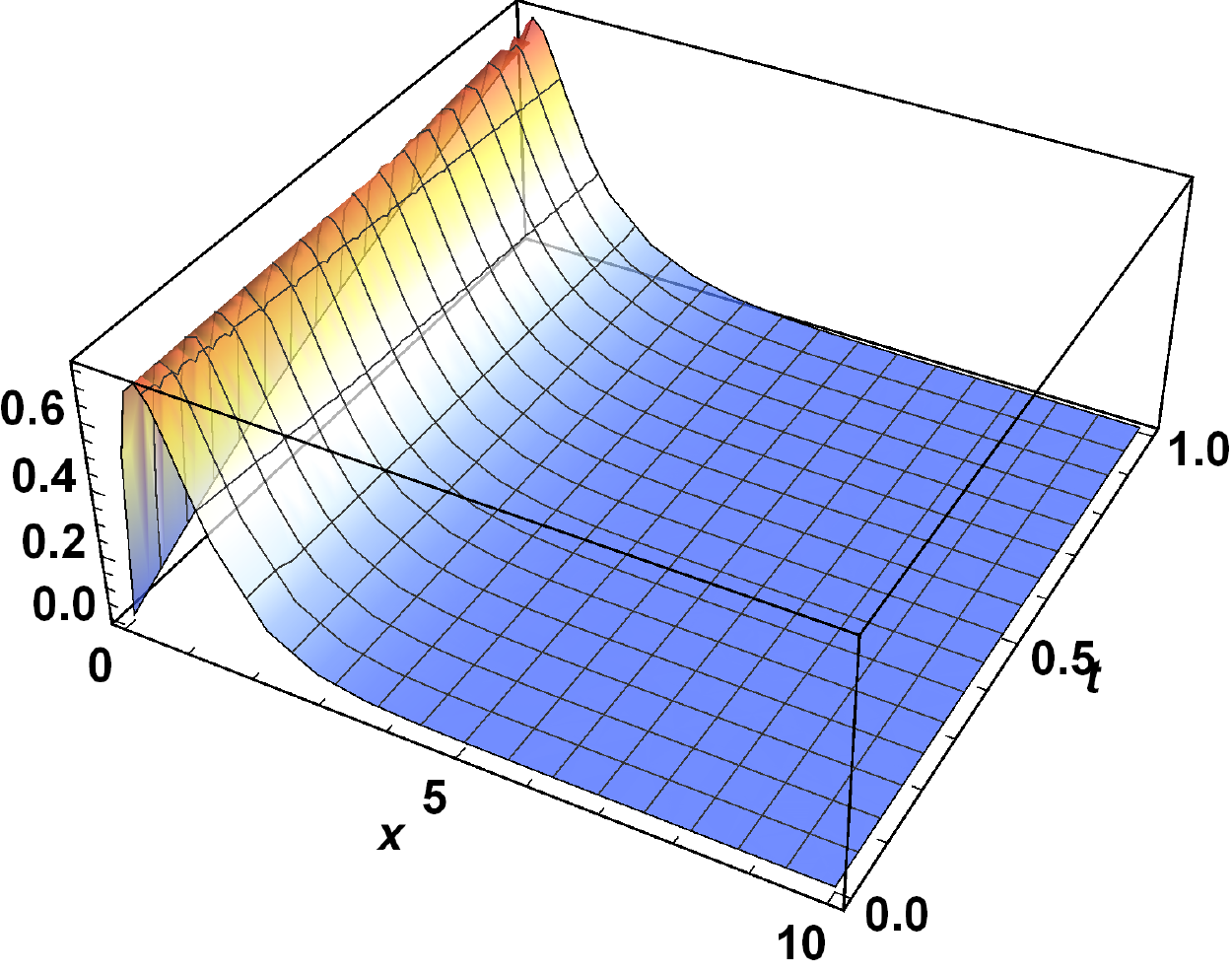}}
 	\subfigure[Time distribution ]{\includegraphics[width=0.45\textwidth,height=0.35\textwidth]{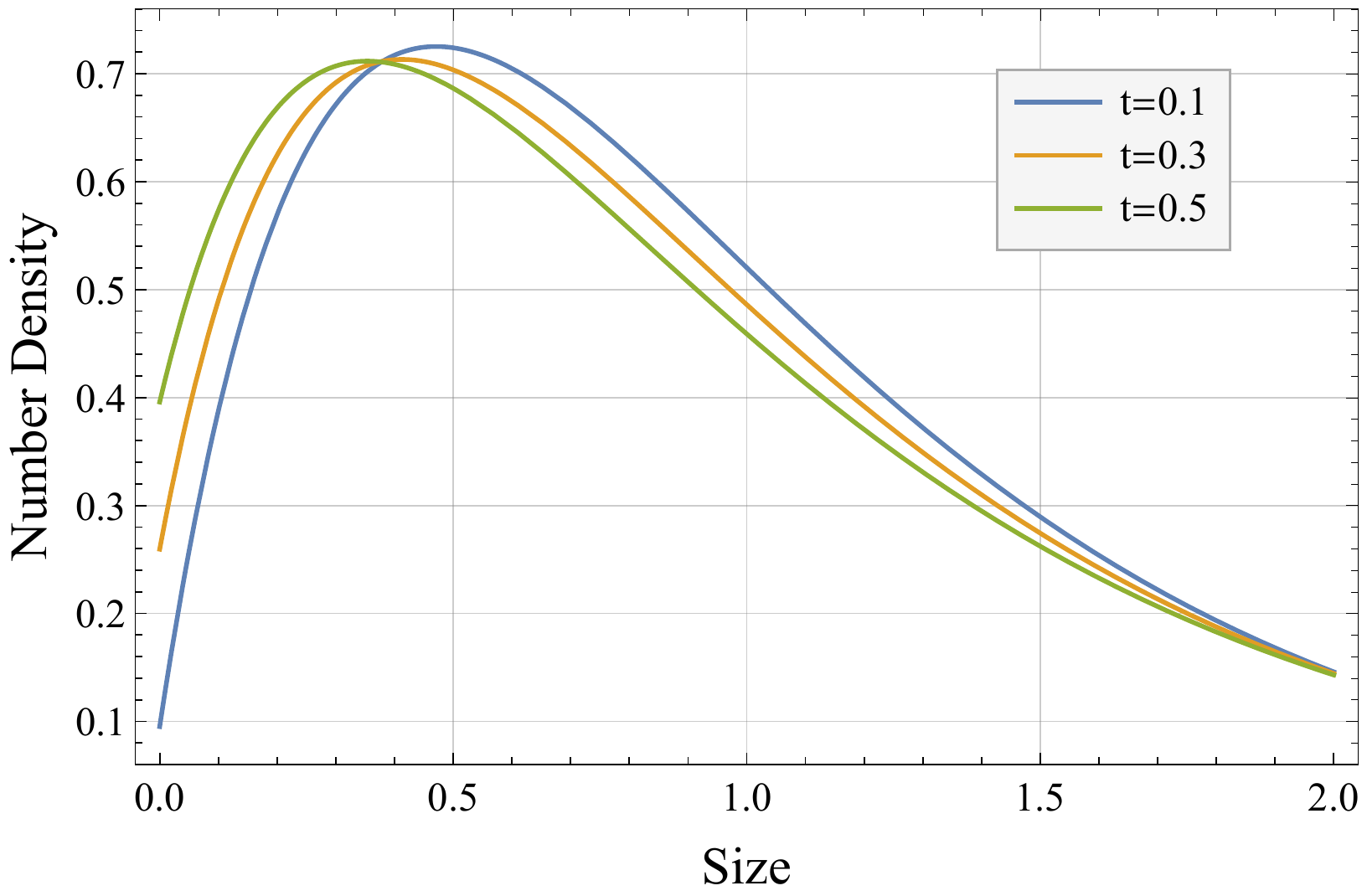}}
\caption{Number density}
\label{fig10}
\end{figure}
\begin{figure}[htb!]
 		\centering
 		\subfigure[Truncated error  ]{\includegraphics[width=0.45\textwidth,height=0.35\textwidth]{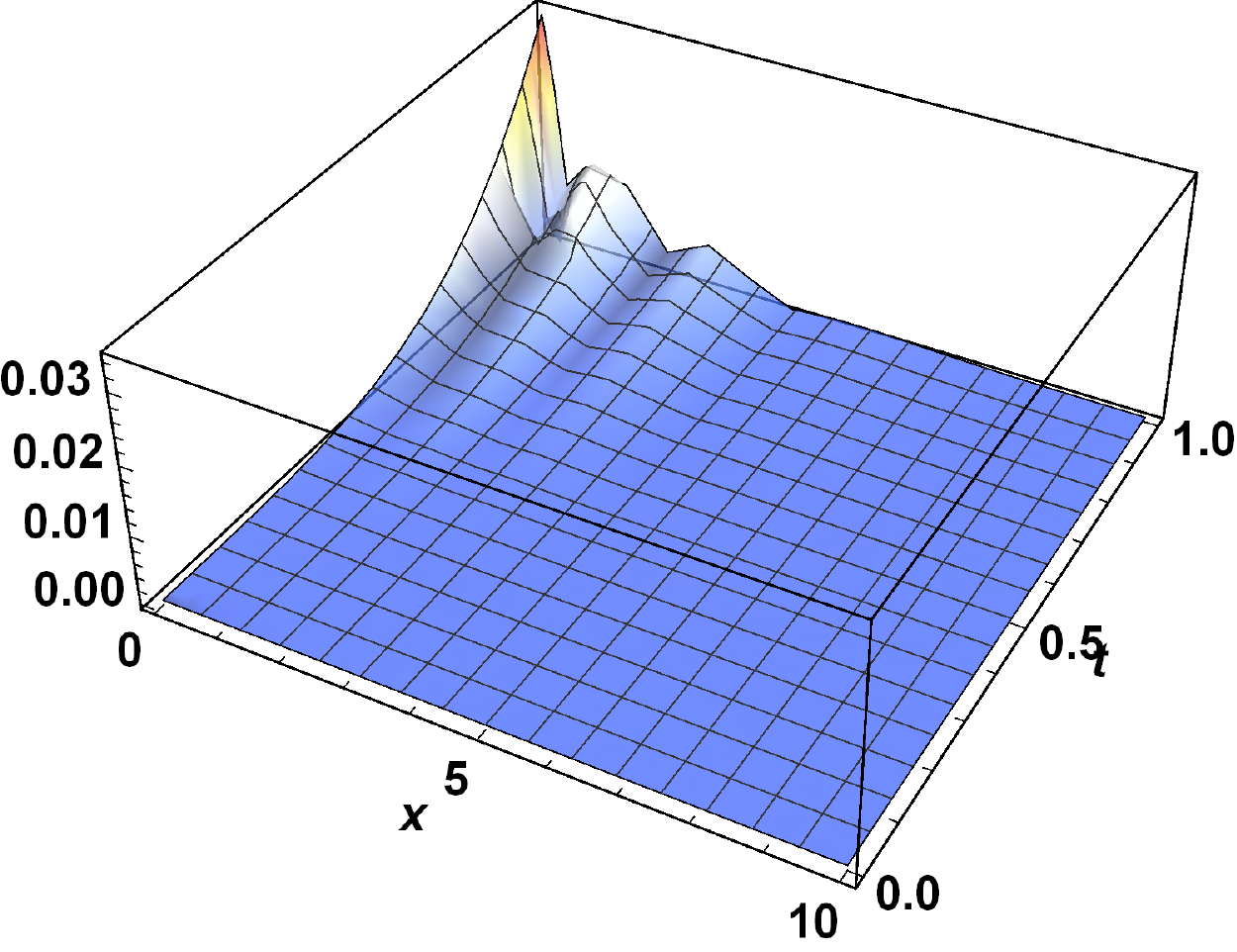}}
 			\subfigure[ Error at $t=0.5$  ]{\includegraphics[width=0.45\textwidth,height=0.35\textwidth]{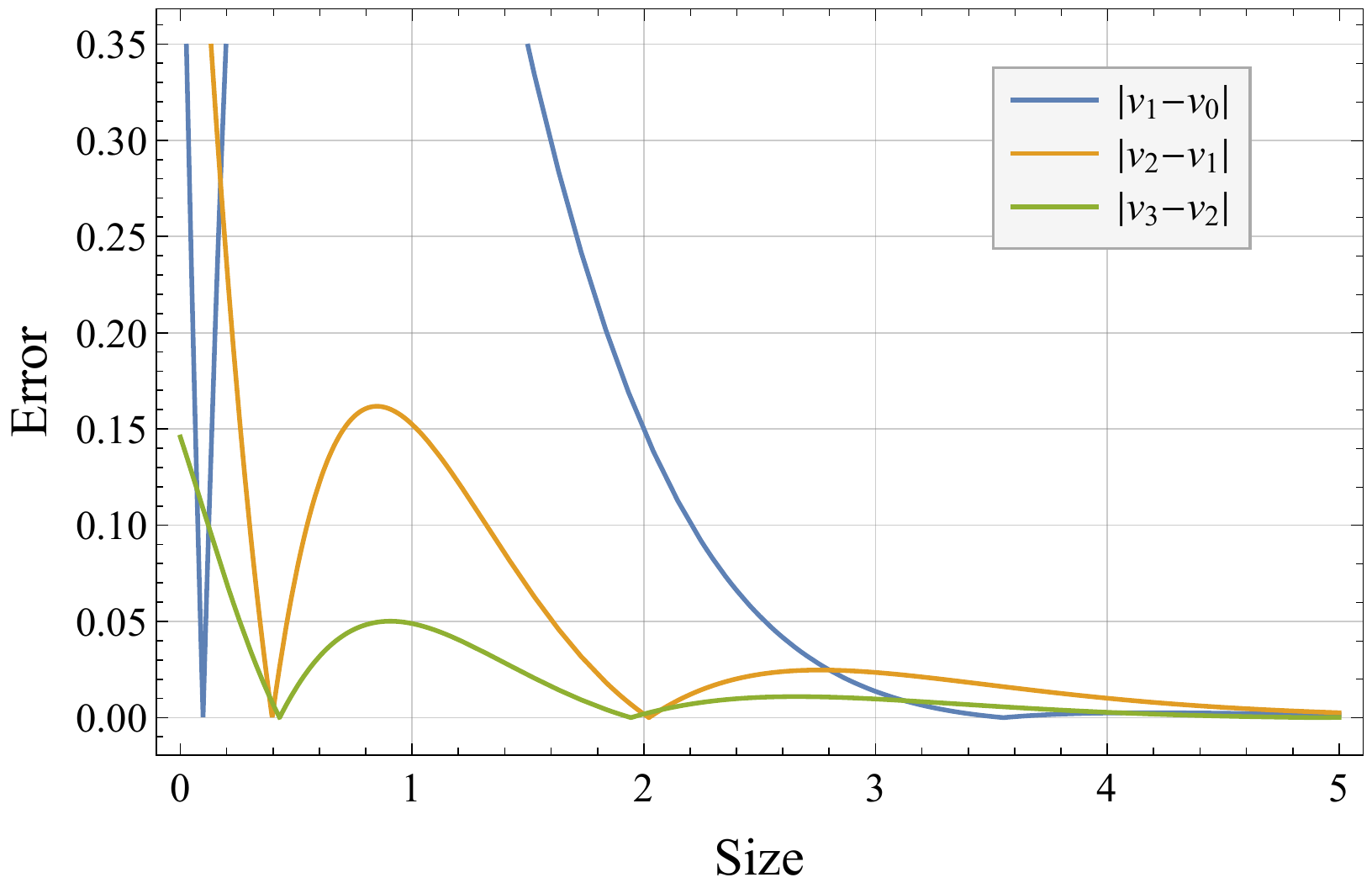}}
 	\subfigure[Zeroth moment]{\includegraphics[width=0.45\textwidth,height=0.35\textwidth]{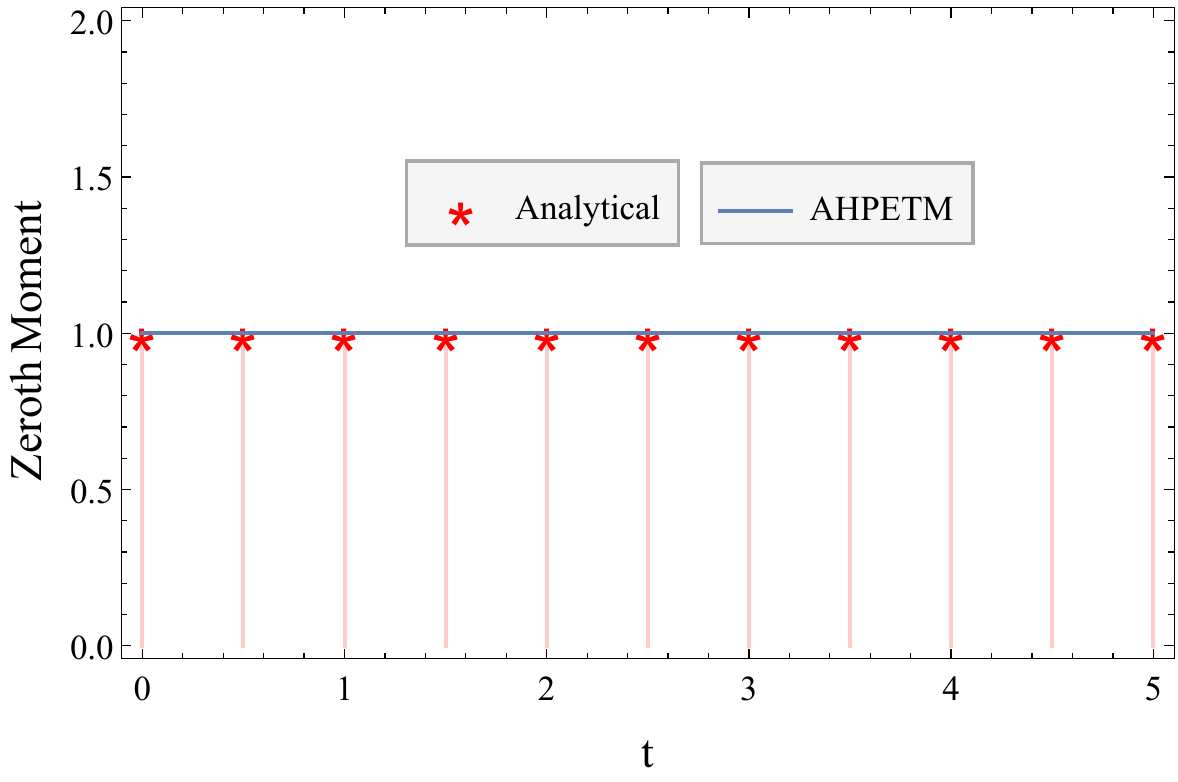}}
\caption{ Error and moment}
\label{fig10a}
\end{figure}
A four-term truncated solution is computed with the aid of "MATHEMATICA". At a specific period $t$, the number density of particles of size $x$ is shown in Figure \ref{fig10}(a). It is observed from Figures \ref{fig10}(a) and (b) that smaller particles tend to increase as time goes on, while larger particles start to fragments into smaller ones. The error between the exact and truncated solutions is presented in Figure \ref{fig10a}(a) and is found to be nearly insignificant. Further, Figure \ref{fig10a}(b) gives the absolute difference between the subsequent components of the series solution and it is clear that the difference between the second and third terms nearly vanishes, which serves as the inspiration for the decision to truncate the solution for three terms. As shown in Figure \ref{fig10a}(c), the truncated solution exhibits steady state behavior for the number of particles as the zeroth moment is constant. This behavior was also analyzed analytically in \cite{lage2002comments}, and thus demonstrating the method's novelty.
\begin{example}
Consider another example of the coupled aggregation-breakage equation \eqref{agg_break_eqn} with the same parameters as taken in Example 6.4 but with selection rate $S(x)=2x $ and initial condition $u(x,0)= 32 x e^{-4 x}$. Similar to the previous case, here as well, steady state behavior of zeroth moment was studied in \cite{lage2002comments}.
\end{example}
Thanks to the formula \eqref{agg_break_iterations}, three terms of the truncated solution are computed as 
$$v_0(x,t)=32 x e^{-4 x}, \quad  v_1(x,t)=\frac{8}{3} t e^{-4 x} \left(32 x^3-24 x^2-12 x+3\right),$$
\begin{align*}
v_2(x,t)=&\frac{8}{945} t^2 e^{-4 x} \bigg(1024 t x^7-3584 t x^6-2688 t x^5+13440 t x^4-3360 t x^3-5040 t x^2+1260 t x\\&+8064 x^5-20160 x^4-7560 x^3+18900 x^2+945 x-945\bigg).
\end{align*}
\begin{figure}[htb!]
 		\centering
 	\subfigure[Number density  $(n=4)$]{\includegraphics[width=0.45\textwidth,height=0.35\textwidth]{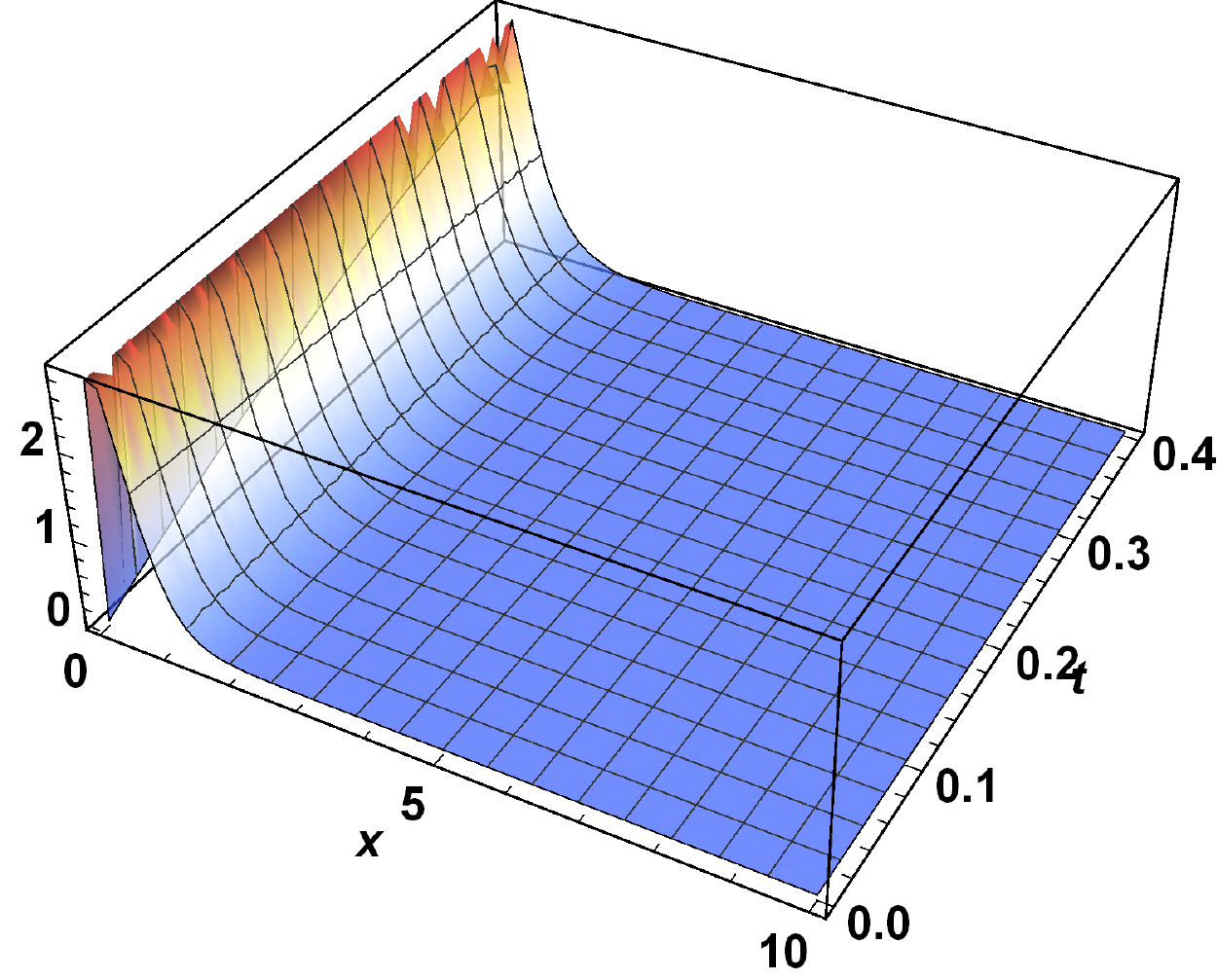}}
 	\subfigure[Time distribution ]{\includegraphics[width=0.45\textwidth,height=0.35\textwidth]{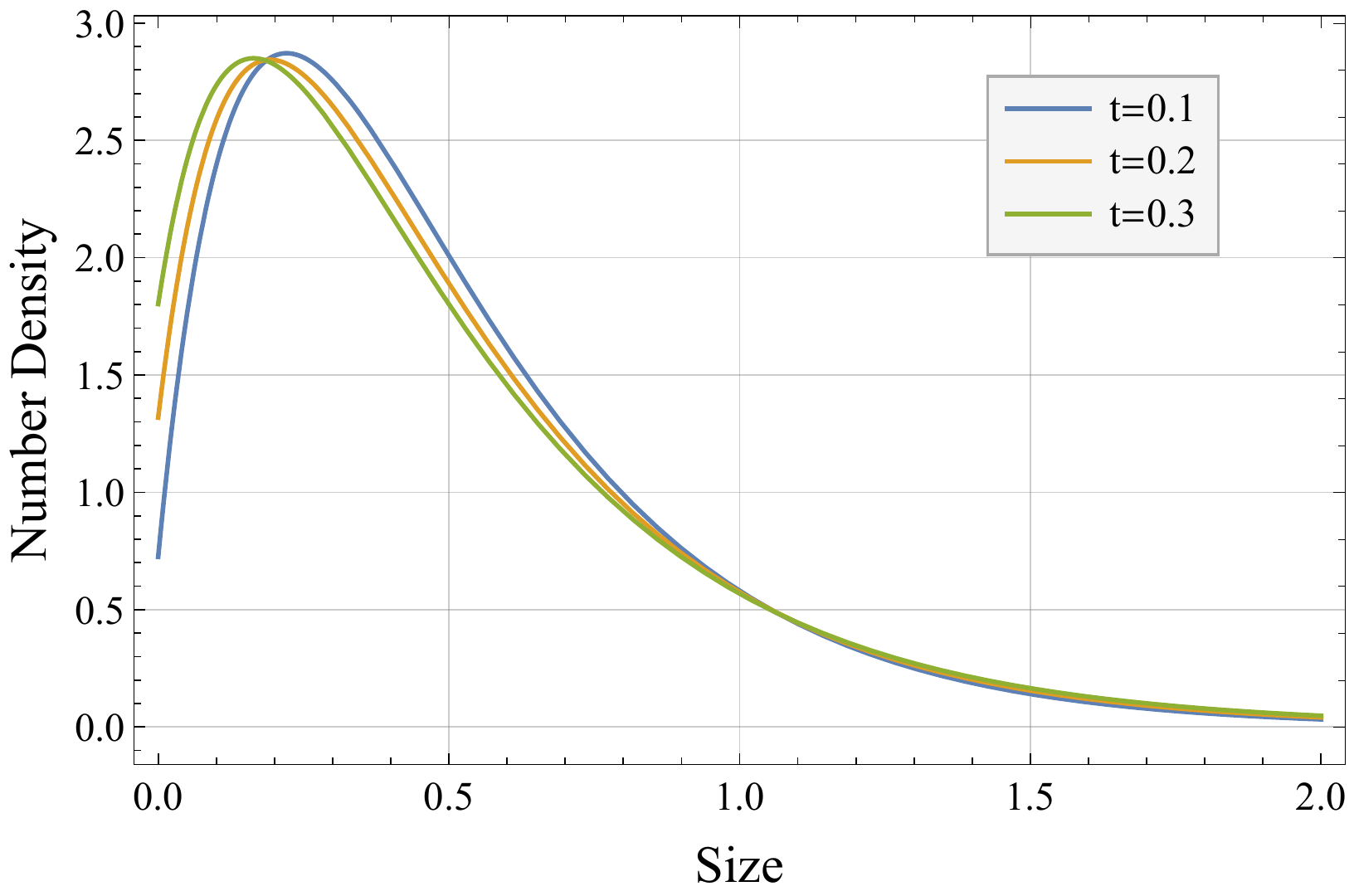}}
\caption{Number density}
\label{fig11}
\end{figure}
\begin{figure}[htb!]
 		\subfigure[Truncated error  ]{\includegraphics[width=0.45\textwidth,height=0.35\textwidth]{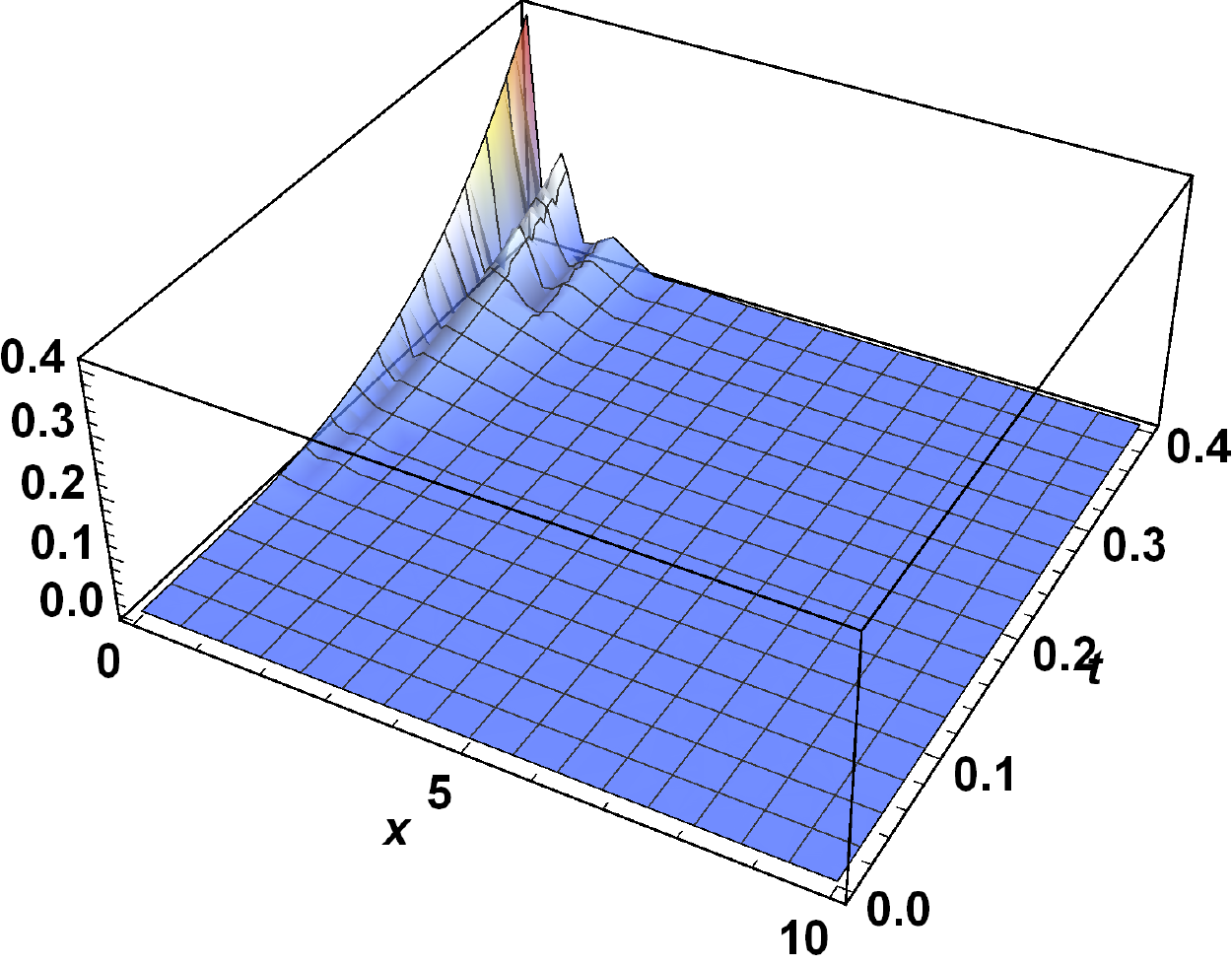}}
 			\subfigure[ Error at $t=0.5$  ]{\includegraphics[width=0.45\textwidth,height=0.35\textwidth]{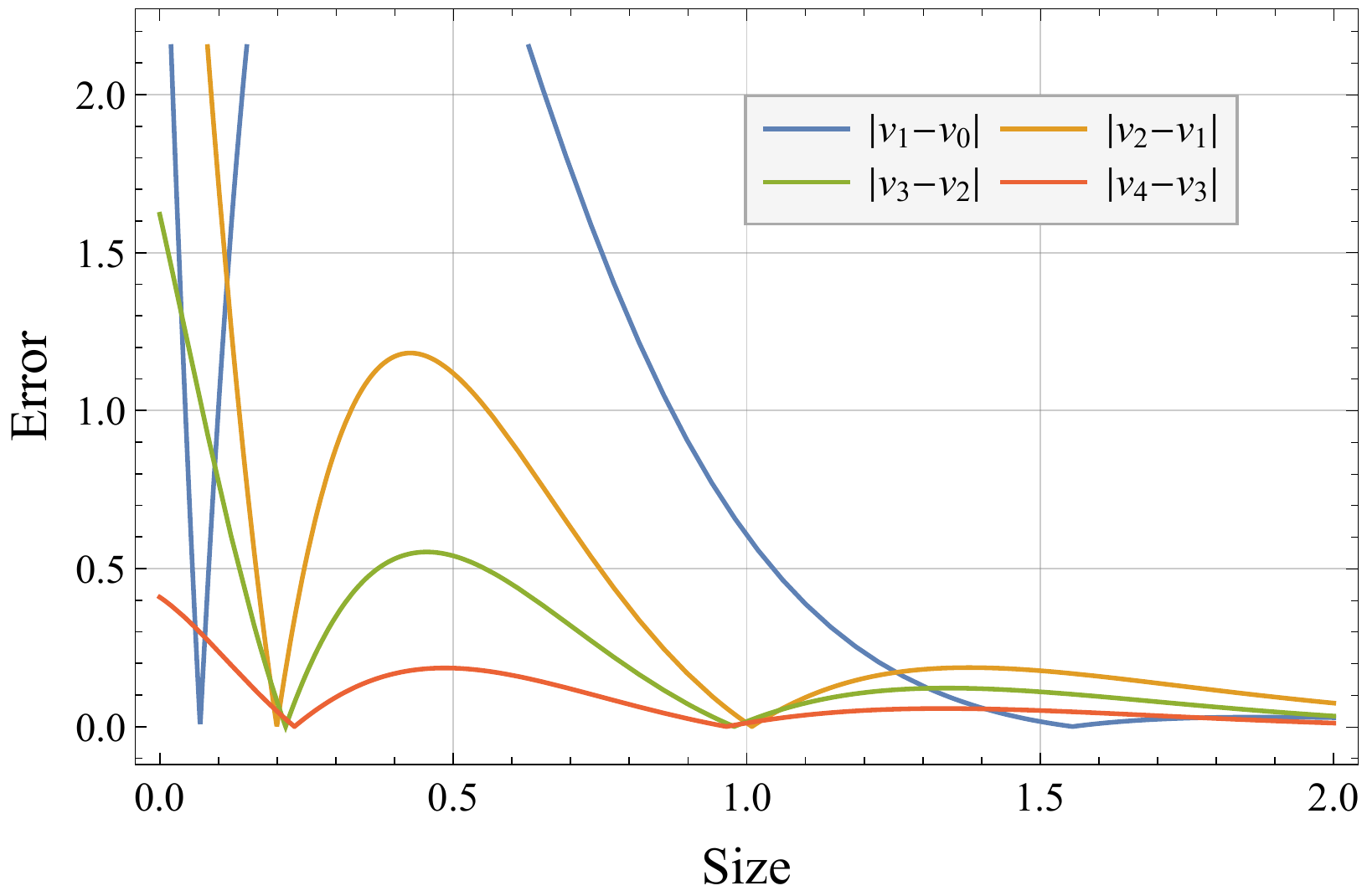}}
 	\subfigure[Zeroth moment]{\includegraphics[width=0.45\textwidth,height=0.35\textwidth]{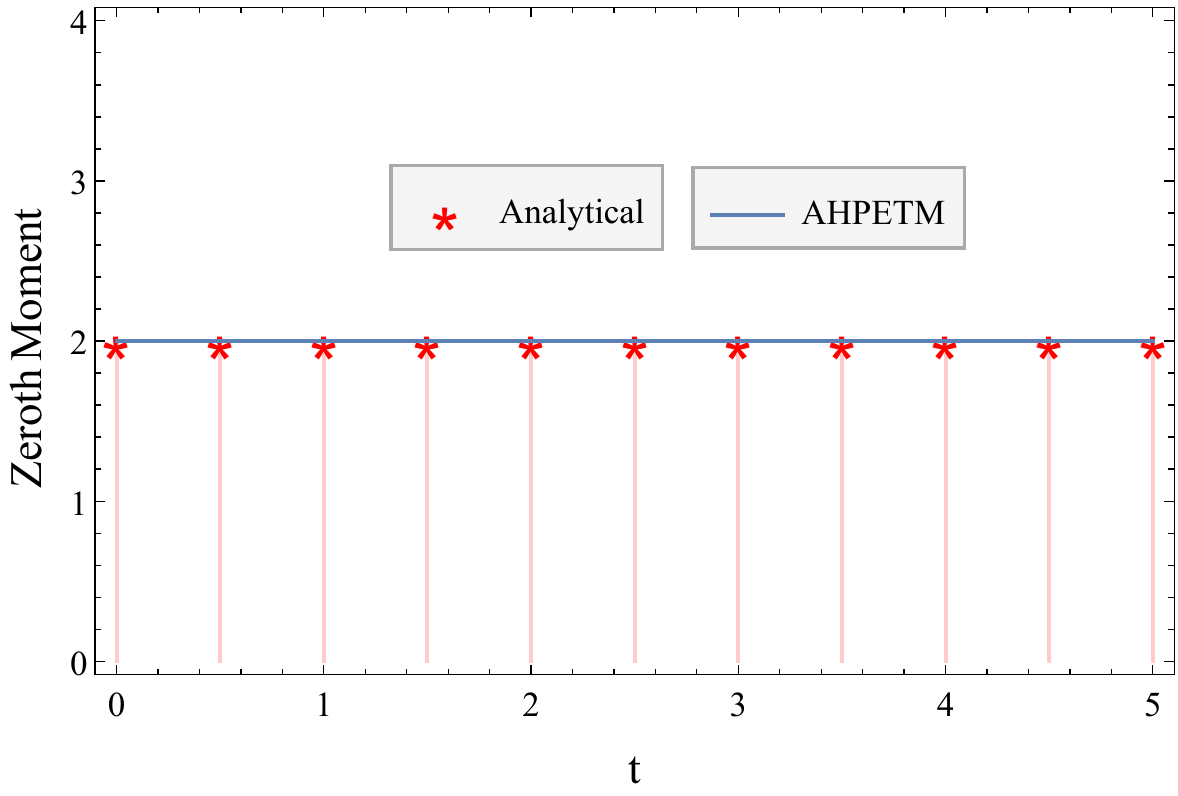}}
\caption{ Error and moment}
\label{fig11a}
\end{figure}
Due to the complexity involved in the terms, a four-term truncated solution is considered. The number density of particle size $x$ in the system is presented in Figure \ref{fig11}(a). Further, Figure \ref{fig11}(b) presents the concentration of particles at different time levels, and an increment in smaller particles is encountered, where as larger particles start breaking as time increases. In Figures \ref{fig11a}(a) and \ref{fig11a}(b), the difference between the consecutive terms is presented, and the error between the second and third terms seems to be vanishing, which leads us to truncate the solution for three terms. As expected, in Figure \ref{fig11a}(c) AHPETM shows the steady-state nature of zeroth moment and is exactly matching with the precise total number of particles.
\subsection{Bivariate Aggregation Equation}
\begin{example}
Let us take two-dimensional aggregation equation \eqref{2DPBE} with the constant aggregation kernel $K(x,x',y,y')=1$ and the initial condition $u(x,y,0)=\frac{16 N_0 x y }{m_1^2 m_2^2} \exp \{-\frac{2 x}{m_1}-\frac{2 y}{m_2}\}$ where the parameters and the exact solution are given in Table \ref{1}. For more details, readers may refer to \cite{singh2020discrete}.  
\end{example}
\begin{table}[h]
\caption{Parameters and exact solution}
\begin{center}
\begin{tabular}{c c } \toprule
$N_0$, $p_1$, $p_2$ & 1 \\\\
$m_1$, $m_2$ & 0.04 \\\\
$u(x,t)$ & $\frac{\left(4 N_0\right)}{\left(m_1 m_2\right) (t+2)^2} \left(\left(p_1+1\right){}^{p_1+1} \left(p_1+1\right){}^{p_2+1}\right) \exp \left(-\frac{\left(p_1+1\right) x}{m_1}-\frac{\left(p_2+1\right) y}{m_2}\right)$ \\ & $\sum _{k=0}^{\infty } \frac{\left(\frac{t}{t+2}\right)^k \left(\left(\left(p_1+1\right){}^{p_1+1}\right){}^k \left(\left(p_2+1\right){}^{p_2+1}\right){}^k \left(\frac{x}{m_1}\right){}^{(k+1) \left(p_1+1\right)-1} \left(\frac{y}{m_2}\right){}^{(k+1) \left(p_2+1\right)-1}\right)}{\Gamma \left(\left(p_1+1\right) (k+1)\right) \Gamma \left(\left(p_2+1\right) (k+1)\right)}$ \\
      \bottomrule
\end{tabular}
\end{center}
\label{1}
\end{table}
The first three elements of the series solution are provided below using the iterations specified in equation \eqref{2dagg_iterations}
$$v_0(x,y,t)= u_0(x,y,0), \quad v_1(x,y,t)=5.42535\times 10^{11} t x y e^{-50 x-50 y} \left( x^2 y^2-0.1152 \times 10^{-4}\right)$$
\begin{align*}
v_2(x,y,t)=& 3.10441\times 10^{-10} x y e^{-50 x-50 y} \bigg(8.06248\times 10^{27} t^3 x^6 y^6-9.10222\times 10^{24} t^3 x^4 y^4\\&+8.73813\times 10^{20} t^3 x^2 y^2-3.35544\times 10^{15} t^3+1.36533\times 10^{25} t^2 x^4 y^4-2.62144\times 10^{21} t^2 x^2 y^2\\& +1.50995\times 10^{16} t^2+6 t\bigg).
\end{align*}
\begin{figure}[htb!]
 		\centering
 	\subfigure[Number density  $(n=4)$]{\includegraphics[width=0.45\textwidth,height=0.35\textwidth]{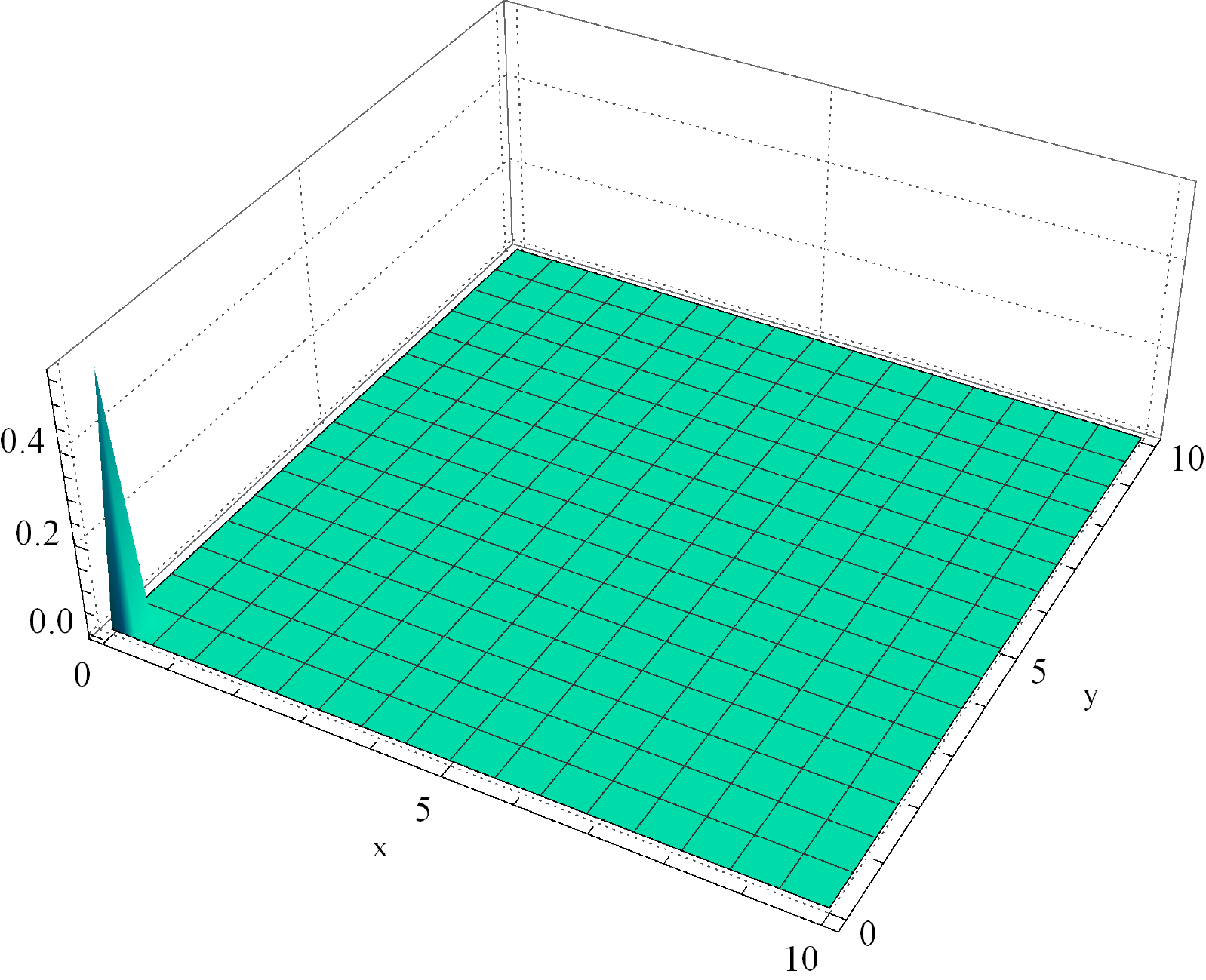}}
 	\subfigure[Truncated error ]{\includegraphics[width=0.45\textwidth,height=0.35\textwidth]{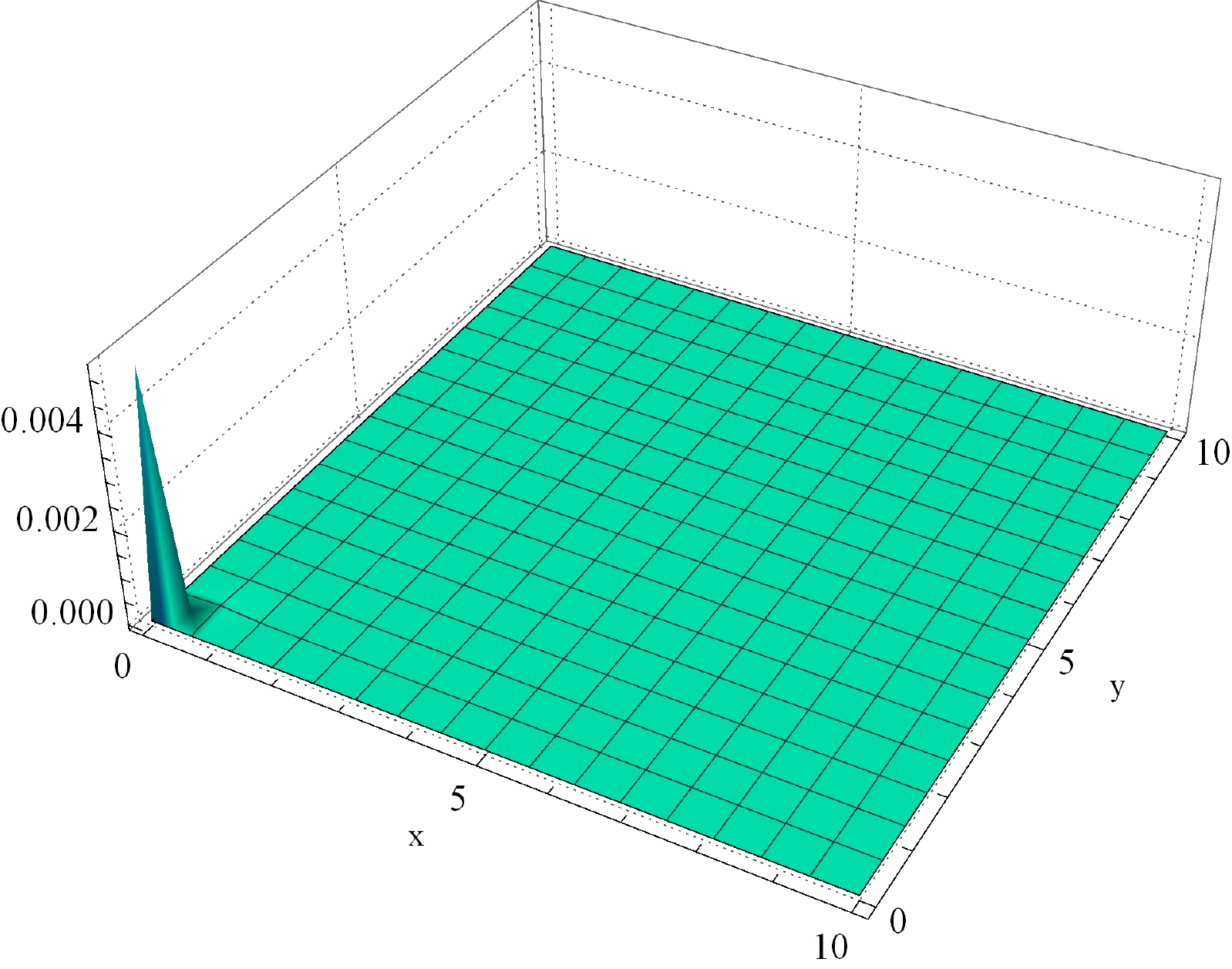}}
 		\subfigure[$|\Phi_2(x,y,t)-u(x,y,t)|$  ]{\includegraphics[width=0.45\textwidth,height=0.35\textwidth]{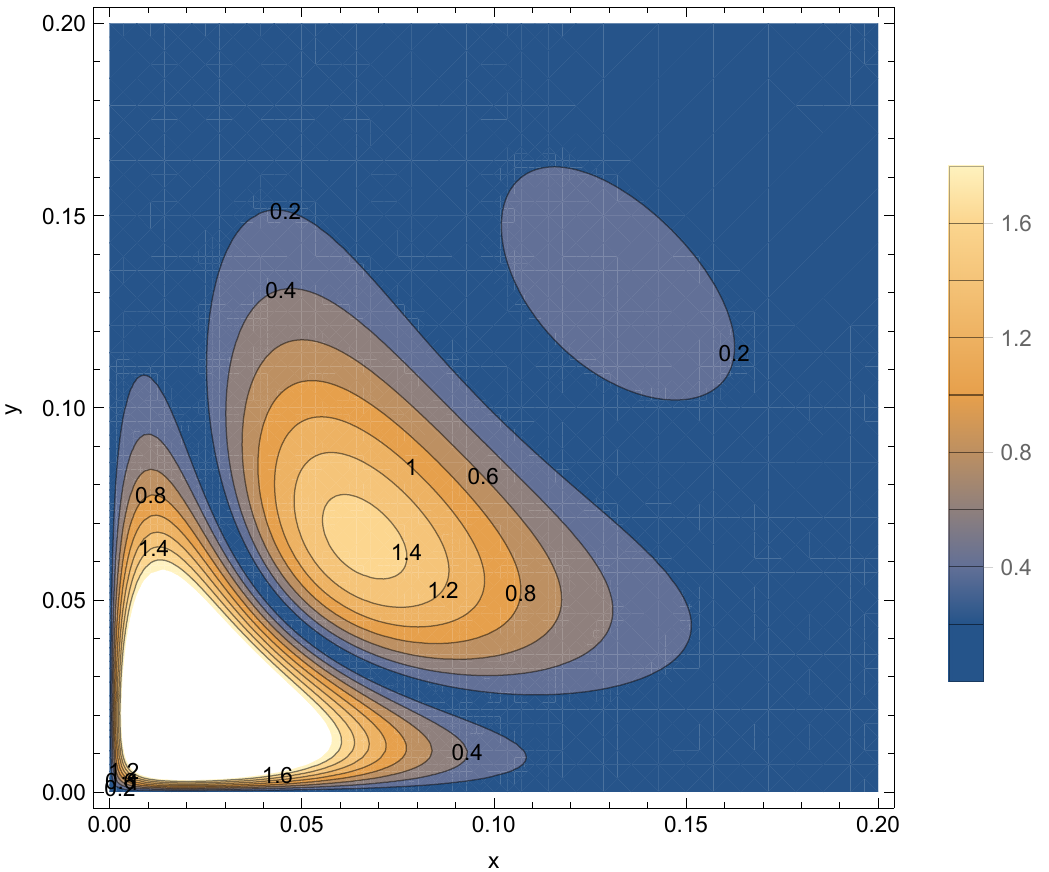}}
 			\subfigure[ $|\Phi_3(x,y,t)-u(x,y,t)|$   ]{\includegraphics[width=0.45\textwidth,height=0.35\textwidth]{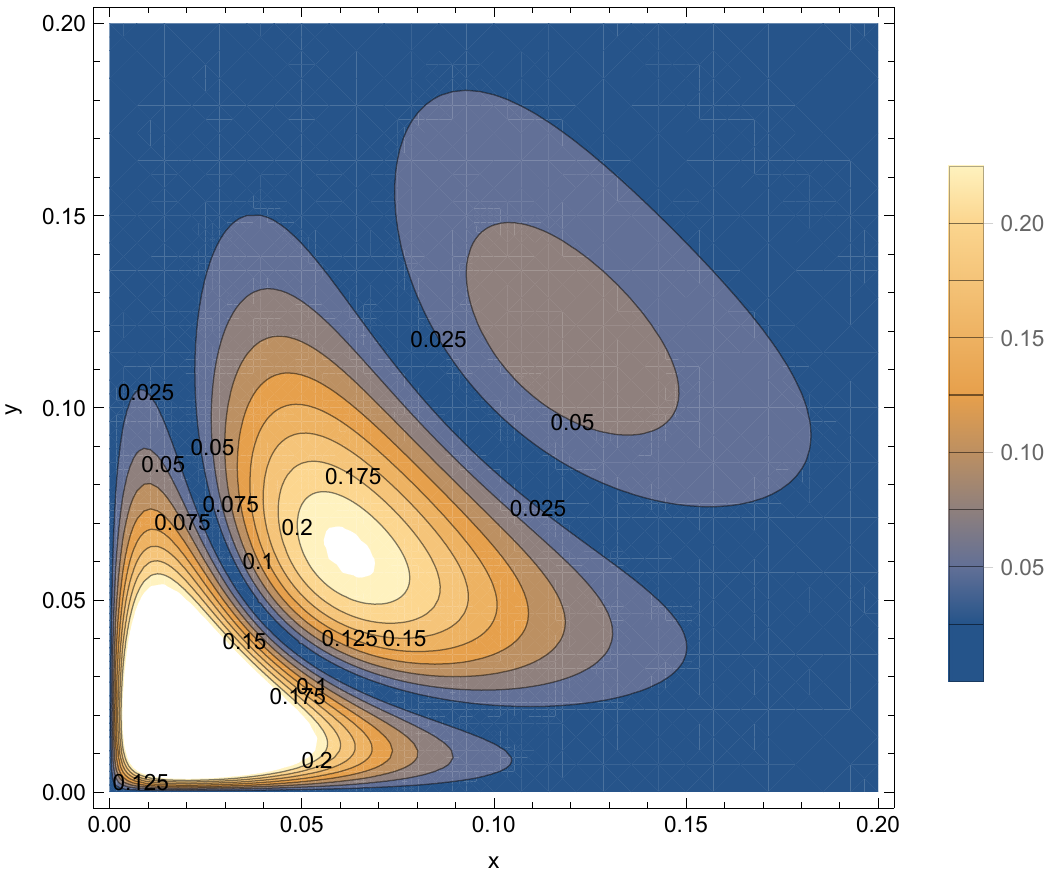}}
 				\subfigure[ $|\Phi_4(x,y,t)-u(x,y,t)|$   ]{\includegraphics[width=0.45\textwidth,height=0.35\textwidth]{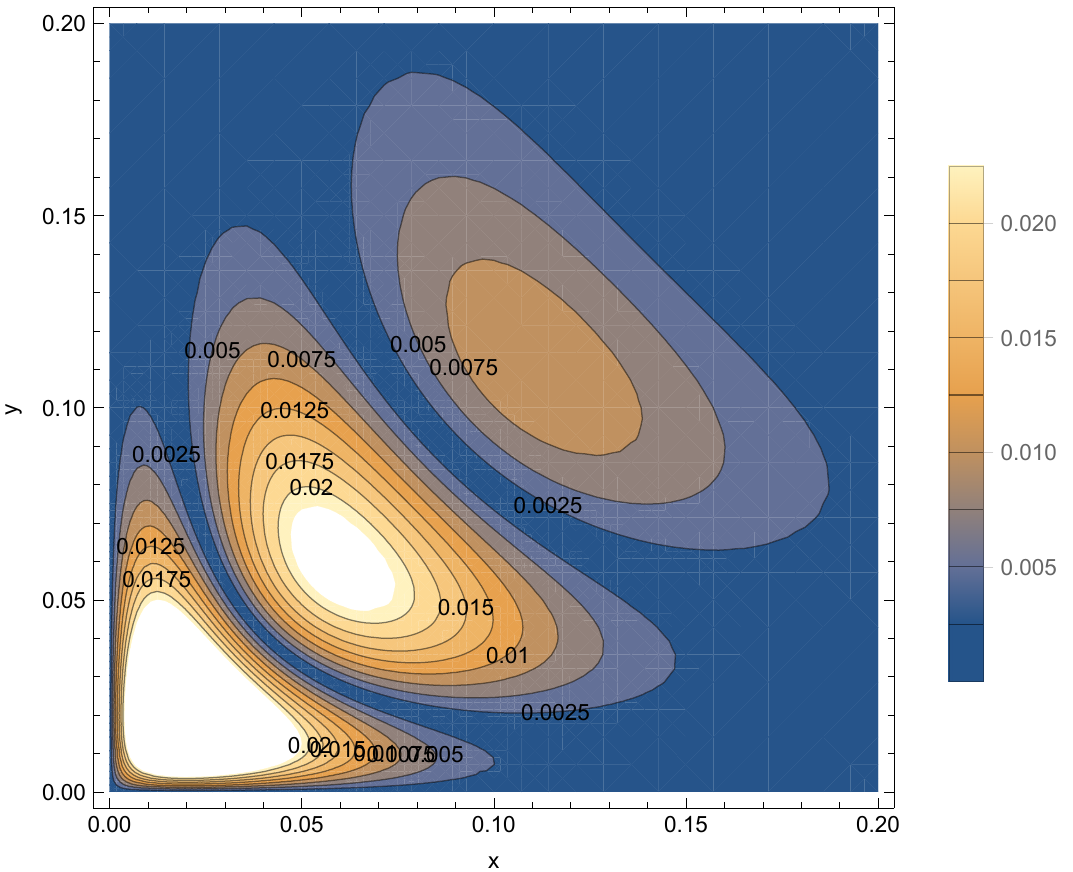}}
 	\subfigure[ Moments]{\includegraphics[width=0.45\textwidth,height=0.35\textwidth]{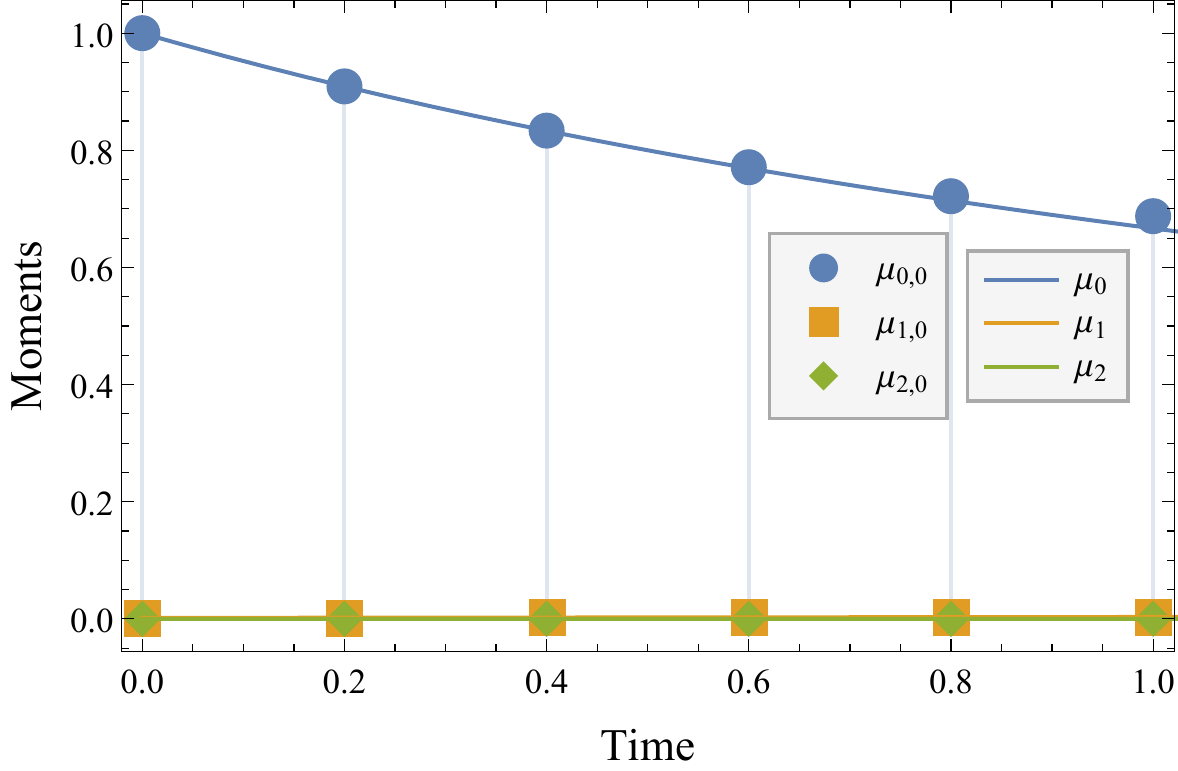}}
\caption{Number density, error and moments}
\label{fig12}
\end{figure}
Continuing in a similar pattern, a four-term truncated solution is computed and compared with the exact solution. Figure \ref{fig12}(a) gives the number density at time $t=0.4$ and it is marked that larger particles almost disappear, and microscopic particles dominate the system. A minimal error is seen between the exact and truncated solutions, according to the error curve shown in Figure \ref{fig12}(b). In addition to this, Figures \ref{fig12}(c)- \ref{fig12}(e) present the contour plots of the errors by taking two, three, and four terms truncated series solutions. One can observe that as the number of terms increases, the error reduces significantly. Finally, Figure \ref{fig12}(f) shows that the approximated moments, namely $\mu_{0,0},\mu_{1,0},\mu_{2,0}$, provide a great agreement with the corresponding exact moments.
	\section{Concluding remarks}
	This study employed AHPETM to solve the fragmentation, multi-dimensional coagulation, and linked aggregation-fragmentation equations. Due to the complexity in the models, convergence analysis were discussed for fragmentation and multi-dimensional aggregation equations considering the constant kernels. With the help of MATHEMATICA, this article also contained the detailed numerical investigations for each of the predefined models. It was observed that, for pure fragmentation equation which is linear, all the schemes offered the same results. However, for non-linear aggregation equation, AHPETM significantly outperformed the results of ADM, HAM, HPM and ODM even after a lengthy period of time. This justified the method's reliability and applicability. AHPETM was also designed to solve non-linear 2-D aggregation and combined aggregation-fragmentation equations due to the accuracy and efficiency observed in the pure aggregation equation and remarkable results were obtained in each case.			
				\clearpage
         						\bibliography{Ref}
         						\bibliographystyle{iEEEtran}
	\end{document}